\documentclass[11pt]{amsart}
\usepackage{amssymb,amsmath,amsthm}
\usepackage{graphicx}
\usepackage[usenames,dvipsnames]{color}
\usepackage{float}
\usepackage{graphicx}
\usepackage{subfigure}
\usepackage{caption}
\usepackage{pgfplotstable}
\usepackage{booktabs}
\usepackage{cite}
\usepackage{url}
\usepackage{caption}
\usepackage{epstopdf}
\usepackage{hyperref}

\numberwithin{equation}{section}

\DeclareMathOperator{\bigo}{O}
\DeclareMathOperator{\Ai}{Ai}
\DeclareMathOperator{\Bi}{Bi}
\DeclareMathOperator{\WKB}{WKB}
\DeclareMathOperator{\Airy}{Airy}
\DeclareMathOperator{\Asymp}{Asymp}

\newtheorem{theorem}{Theorem}

\newtheorem{remark}[theorem]{Remark}

\newtheorem{proposition}[theorem]{Proposition}

\numberwithin{equation}{section}
\numberwithin{theorem}{section}

\newcommand{\R}{\mathbb{R}}
\newcommand{\C}{\mathbb{C}}

\newcommand{\abs}[1]{\left\vert#1\right\vert}

\newcommand{\Id}{\mathcal{I}}

\newcommand{\paren}[1]{\left(#1\right)}
\newcommand{\bracket}[1]{\left[#1\right]}
\newcommand{\set}[1]{\left\{#1\right\}}

\newcommand{\eps}{\epsilon}
\newcommand{\ximax}{\xi_{\max}}
\newcommand{\xmax}{x_{\max}}

\title[Singular profiles of a DNLS equation]{Local structure of singular profiles for a
  Derivative Nonlinear Schr\"odinger Equation}

\author[Cher]{Yuri Cher}
\address{Department of Mathematics, University of Toronto, Toronto ON,
  M5S 2E4 Canada,  \ ycher@math.toronto.edu}

\author[Simpson]{Gideon Simpson}
\address{Department of Mathematics, Drexel Unixersity, Philadelphia PA, 19104, \  simpson@math.drexel.edu}
 \thanks{G.S. is partially supported by NSF through grant number
   DMS-1409018.  Work reported here was run on hardware supported by Drexel's University Research Computing Facility}

\author[Sulem]{Catherine Sulem}
\address{Department of Mathematics, University of Toronto, Toronto ON,
  M5S 2E4 Canada,  \ sulem@math.toronto.edu}
 \thanks{C.S. is  partially  supported by NSERC through grant
  number 46179--13.}


\date{\today}

\begin{document}

\maketitle

\begin{abstract}
  The Derivative Nonlinear Schr\"odinger equation is an $L^2$-critical
  nonlinear dispersive equation model for Alfv\'en waves in space
  plasmas.  Recent numerical studies \cite{Liu:2013ej} on an
  $L^2$-supercritical extension of this equation provide evidence of
  finite time singularities. Near the singular point, the solution is
  described by a universal profile that solves a nonlinear elliptic
  eigenvalue problem depending only on the strength of the
  nonlinearity.  In the present work, we describe the deformation of
  the profile and its parameters near criticality, combining
  asymptotic analysis and numerical simulations.

\end{abstract}

\section{Introduction}
The derivative nonlinear Schr\"odinger (DNLS) equation

\begin{align}\label{eqn:DNLS_original}
  \begin{cases}
    i u_t + u_{xx} + i \left(|u|^2 u\right)_x =0 ,  \quad x\in \mathbb{R} \\
    u(x,0) = u_0(x).
  \end{cases}
\end{align}
is a canonical equation arising from the Hall-Magnetohydrodynamics
equations. It appears in the context of Alfv\'en waves propagating
along an ambient unidirectional magnetic field in a long wavelength
regime \cite{Sulem:1999kx}. More recently, it was used to model rogue waves and plasma
turbulence \cite{SanchezArriaga:2010jy}.  Under the gauge
transformation,
\begin{align}
  \label{e:gauge}
  \psi(x,t) = u(x,t) \exp\set{\frac{ i}{2} \int_{-\infty}^{x}
  |u(y,t)|^2 dy },
\end{align}
\eqref{eqn:DNLS_original} becomes
\begin{align}
  \label{eqn:DNLS_gauged}
  i \psi_t + \psi_{xx} +i |\psi|^2 \psi_x =0.
\end{align}
Eq. \eqref{eqn:DNLS_gauged} has appeared as a model for ultrashort optical
pulses, \cite{Anderson:1983wg,Moses:2007vv,Tzoar:1981vq}.

Solutions to the DNLS equation exist locally in time in $H^1(\R)$ and
they can be extended for all time if the initial conditions are
sufficiently small in $L^2$, namely $\|u_0\|_2 < \sqrt{2\pi}$
\cite{Hayashi:1992wl,Hayashi:1993vj}.  The global in time result
relies on two invariants of the equation,
\begin{align}
  \text{Mass:}\quad M[u] & \equiv \int \abs{u}^2 ~dx,\\
  \text{Hamiltonian:}\quad H[u]&\equiv \int \left(\abs{u_x}^2 + \tfrac{3}{2} \Im (
                                 |u|^2 u\bar{u}_x ) + \tfrac{1}{6} |u|^6 \right) dx
\end{align}
and the sharp constant in a Gagliardo-Nirenberg inequality.  Very
recently, Wu \cite{Wu:2014uc} 
showed that the upper bound on the $L^2$-norm of the initial
conditions can be increased to $\| u_0 \|_2 < \sqrt{4\pi}$ using the
conservation of momentum
\begin{equation}\label{mom}
  \text{Momentum:}\quad I[u]\equiv \int  \left( \Im(\bar{u}u_x) - \tfrac{1}{2} 
    |u|^4\right) dx
\end{equation}
and a different Gagliardo-Nirenberg inequality.  As discussed below,
DNLS has a two-parameter family of solitary waves that decay
exponentially fast at infinity (bright solitons) as well as algebraic
solitons (lumps). It is interesting to notice that $\sqrt{4\pi}$ is
the $L^2$-norm of the lump soliton (see \eqref{LumpSigma} with
$\sigma=1$).  Furthermore, DNLS is completely integrable via the
inverse scattering transform \cite{Kaup1978} and has an infinite
number of conserved quantities. Recent works using the inverse
scattering method provide global solutions for initial conditions in a
spectrally determined (open) subset of weighted Sobolev spaces
containing a neighborhood of zero,
\cite{Pelinovsky:2015aa,Liu:2015aa}.  Global well-posedness for large
data remains an open problem.

Equation \eqref{eqn:DNLS_gauged}, along with
\eqref{eqn:DNLS_original}, is invariant to the scaling transformation
$ \psi \mapsto \psi_\lambda = \lambda^{-\frac{ 1}{2}}
\psi(\lambda^{-1} x, \lambda^{-2} t)$.
It is $L^2$-critical in the sense that
$\|\psi_\lambda\|_{L^2} = \|\psi\|_{L^2}$ and has the same scaling
properties as the focusing nonlinear Schr\"odinger equation,
\begin{equation}\label{NLS}
  i u_t + \Delta u + \abs{u}^{2\sigma} u = 0, \quad u: (x,t) \in  \R^{d}\times \R \to \C,
\end{equation}
with $d\sigma =2$.  However, it has very different structural
properties, such as the aforementioned integrability.  In contrast, it
is well known that for the $L^2$-critical and supercritical NLS
equations ($d\sigma\ge 2$), blowup occurs for initial conditions with
$L^2$-norm exceeding that of the ground state.

In the context of dispersive equations, the comparative study of
equations with critical and supercritical nonlinearities has been very
fruitful, \cite{MRS2011,Sulem:1999kx}. From this perspective, and to
gain additional insight into the properties of solutions to the DNLS
equations, a generalization of \eqref{eqn:DNLS_gauged} was introduced,
\begin{align}\label{eqn:gDNLS}
  i\psi_t + \psi_{xx} + i |\psi|^{2\sigma} \psi_x = 0,
\end{align}
that we will refer to as ``gDNLS'', \cite{Liu:2013cq,Liu:2013ej}. If
$\sigma >1$, the gDNLS equation is $L^2$-supercritical. 
Recent work  by Hayashi and Ozawa \cite{HO16} shows that it is locally well-posed in $H^1$ and globally well-posed if the initial conditions
are small enough; see  Ambrose and Simpson for related results on the
periodic problem, \cite{Ambrose:2015en}. Numerical
simulations performed in \cite{Liu:2013ej} strongly indicate that
\eqref{eqn:gDNLS} may present finite time singularities when
$\sigma >1$. More specifically, near the singular point, $(x^*, t^*)$,
the solution is locally approximated by
\begin{align}
  \label{e:blowup_asympt}
  \psi(x,t) \approx \left(\frac{1}{2 a (t^* - t)}\right)^{1/4\sigma} Q
  \left(\frac{x -x^*}{\sqrt{2 a (t^* - t)}} + \frac{b}{a}\right) e^{i
  \left(\theta + \frac{1}{2a} \ln \frac{t^*}{t^* - t} \right)}.
\end{align}
The blow-up profile $Q$ is a complex-valued function solving the
nonlinear eigenvalue problem
\begin{align}
  \label{QProfile}
  Q_{\xi \xi} - Q + ia \left(\tfrac{1}{2\sigma} Q+ \xi Q_\xi\right) - ib
  Q_\xi + i |Q|^{2\sigma} Q_\xi = 0.
\end{align}
The coefficient $b$ can be changed, or even eliminated by translating
the independent variable (as long as $a\ne 0$).  It was observed
numerically that the amplitude $|Q|$ of the profile has only one
maximum. In this work, we will choose the coefficient $b$ so that
$\max|Q| $ is at the origin.  The coefficients $a$, $b$, and the
function $Q$ all depend on $\sigma$, but were observed in the
simulations to be universal (up to simple scalings), in the sense that
the same values emerged, regardless of the initial conditions, for the
time dependent problem.

The local structure of $\psi$, near the singularity, can be extracted
using time dependent rescaling.  First, we note that gDNLS is
invariant under the transformation
$\psi \mapsto \psi_\lambda = \lambda^{-\frac{ 1}{2\sigma}}
\psi(\lambda^{-1} x, \lambda^{-2} t).  $
This motivates the introduction of the scaled dependent and
independent variables:
\begin{align}\label{eqn:DynamRescalingTransformation}
  \psi(x,t ) = \lambda(t)^{-\frac{ 1}{2\sigma}} v( \xi, \tau) , \quad
  \xi = \frac{x - x_0(t)}{\lambda(t)}, \quad \tau = \int_{0}^{t} \frac{dt^\prime}{\lambda^2(t^\prime)}.
\end{align}
The scaling factor $\lambda(t)$ is chosen to be proportional to
$\|\psi_x\|_{L^2} ^{-q}$ , $q= 2\sigma/(\sigma+1)$, while the shift
$x_0(t)$ is used to keep the bulk of the solution at the origin. The
rescaled function $v$ satisfies
\begin{align}\label{eqn:DynamRescaledDNLS}
  \begin{cases}
    i v_\tau + v_{\xi\xi} + i \alpha(\tau) \left(\frac{v}{2\sigma} + \xi v_\xi\right) - i\beta(\tau) v_\xi + i |v|^{2\sigma} v_\xi =0 \\
    \alpha = - \lambda \frac{d\lambda}{dt} ,\quad \beta = \lambda
    \frac{dx_0}{dt}.
  \end{cases}
\end{align}
For large $\tau$, it was observed that $v \sim e^{i C \tau} Q(\xi)$
and the parameters $\alpha (\tau)$, $\beta(\tau)$ tend to constant
values independent of the initial conditions. Substituting in
$v \sim e^{i C \tau} Q(\xi)$, canceling out the time harmonic piece,
and applying a simple rescaling turns \eqref{eqn:DynamRescaledDNLS}
into \eqref{QProfile}.  Under the transformation
\eqref{eqn:DynamRescalingTransformation} the conserved quantities
scale like
\[
M(\psi) = \lambda^{1-\frac{1}{\sigma}} M(v), \quad H(\psi) =
\lambda^{-1-\frac{1}{\sigma}} H(v), \quad I(\psi) =
\lambda^{-\frac{1}{\sigma}} I(v).
\]
In particular, since $\lambda \to 0$ as $\tau \to \infty$, we have
that $M(Q) =\infty$ while $I(Q) = H(Q) = 0$. The blow-up profile
equation is reminiscent of the profile describing the singularity
structure of radially symmetric $L^2$-supercritical NLS equations
\eqref{NLS}, with $\sigma d>2$ satisfying
\begin{align}
  S_{\xi\xi} + \tfrac{d-1}{\xi}S_\xi - S + i a\left(\tfrac{1}{\sigma} S+ \xi S_\xi\right) + |S|^{2\sigma}S = 0, \quad \xi>0,
\end{align}
derived in \cite{McPSS} and
studied in \cite{BCR}.
An asymptotic analysis near the critical value $\sigma d \to 2$
\cite{Sulem:1999kx} provides the behavior of the parameter $a$
\begin{align}\label{eqn:NLS_a_relation}
  \sigma d - 2 \propto a^{-1}e^{ -\frac{\pi}{a}}, \quad \sigma d \to 2,
\end{align}
and the profile $Q$ is asymptotically close to NLS ground state.  This
approach was key to predicting the generic blowup of solutions to
critical NLS and its $\log-\log$ correction.

The simulations performed in \cite{Liu:2013ej} suggested that, as the
nonlinearity $\sigma\to 1$, the coefficients $a$ and $b$, viewed as
functions of $\sigma$, behave as $a\to 0$ and $b\to b_0 >0$.  In the
present work, we perform a detailed asymptotic study complemented by
numerical results to describe the deformation of profile $Q$ and the
behaviour of the parameters $a(\sigma)$, $b(\sigma)$ in the
$\sigma \to 1$
limit.  We find that
\begin{align}\label{e:QSimLumpConclusion}
  Q(\xi) \sim L(\xi) \exp\left\{- i\left(\frac{a \xi^2}{4} - \frac{b \xi}{2} + \frac{1}{4} \int_{0}^{\xi} |Q|^{2}\right)\right\}
\end{align}
where $L(\xi) = \sqrt{\frac{8}{1+4\xi^2}}$ is the algebraic soliton of
the DNLS solving
\begin{equation}
  \label{e:cubic_lump}
  L_{\xi\xi} - L^3 + \tfrac{3}{16} L^5 =0.
\end{equation}
The parameters $a(\sigma)$ and $b(\sigma)$ behave like power laws of
$(\sigma - 1)$, namely:
\begin{subequations}
  \label{e:relations}
  \begin{align}
    \label{e:arelation} a &\propto (\sigma - 1)^{\gamma_a}, \ \gamma_a \approx 3.2,\\
    \label{e:brelation} 2-b &\propto (\sigma-1)^{\gamma_b}, \ \gamma_b = 2.
  \end{align}
\end{subequations}

Our paper is organized as follows. In Section 2, we describe basic
properties of the profile $Q$ solution of \eqref{QProfile}.  In
Section 3, we present numerical simulations of
\eqref{QProfile}. Motivated by these calculations, we analyze the
deformation of the profile as $\sigma\to 1$ and connect the behavior
of $Q$ at $\pm \infty$ using asymptotic methods in Section 4.  We
complement it by a careful analysis of the numerical data to predict
relations between the parameters $a$ and $b$ and $\sigma$ in the limit
$\sigma \rightarrow 1$.  In Section 5, we impose the vanishing momentum
condition to extract another relation between the parameters.
Concluding remarks are presented in Section 6.  Finally, in Appendix
A, we give the proof of Proposition \ref{prop:Asymptotics} on the
behavior of the profile $Q(\xi)$ for large $\xi$, and in Appendix B,
we provide details of the numerical methods, in particular how we deal
with solutions that decay slowly at infinity.

\section{Preliminary Results}
We recall basic properties of solutions to the profile equation
\eqref{QProfile}, give a preliminary discussion on the relation
between them and soliton solutions to the gDNLS equation.

\subsection{gDNLS Soliton Solutions}
The gDNLS equation \eqref{eqn:gDNLS} has a two-parameter family of
soliton solutions in the form
\begin{equation*}\label{e:gdnls_soliton_form}
  \psi_{\omega,c} (x,t) 
  = R_{\omega,c}(x-ct) \exp\set{i\Big(\omega t +\frac{c}{2} (x-ct)
    - \frac{1}{2\sigma+2} \int_{-\infty}^{x-ct}
    R^{2\sigma}_{\omega,c} \Big)},
\end{equation*}
where $R_{\omega,c}$, satisifies
\begin{equation}
  \label{e:gdnls_soliton}
  \partial_{\xi\xi}R_{\omega,c}  - \paren{\omega - \tfrac{c^2}{4}}
  R_{\omega,c} - \frac{c}{2}\abs{R_{\omega,c}}^{2\sigma}R_{\omega,c} +
  \tfrac{2\sigma+1}{(2\sigma+2)^2}\abs{R_{\omega,c}}^{4\sigma} R_{\omega,c}=0,
\end{equation}
subject to the boundary conditions $R_{\omega,c}\to 0$ as
$\xi \to \pm \infty$.  Eq. \eqref{e:gdnls_soliton} has smooth, real
valued, solutions expressed in terms of hyperbolic functions for all
$c$ and $\omega >c^2/4$. Without loss of generality, we fix
$\omega=1$, denote $c=b$, and suppress the subscripts.  The equation
for $R$ is then
\begin{align}
  \label{SolitonEqn}
  R_{\xi\xi} - \left(1 - \tfrac{ b^2}{4}\right) R - \frac{b}{2}
  R^{2\sigma+1} + \tfrac{2\sigma+1}{(2\sigma+2)^2} R^{4\sigma+1} = 0.
\end{align}
For $\abs{b}<2$, the solutions are smooth and exponentially decaying,
\begin{align}
  \label{BrightSigma}
  R = B_\sigma \equiv \left(\frac{ (\sigma+1)(4-b^2)}{2(\cosh (\sigma
  \sqrt{4- b^2} \xi) - \frac{b}{2})}\right)^\frac{1}{2\sigma}.
\end{align}
We refer to these as ``bright'' soliton solutions.  In the limit
$b\nearrow 2$, another solution emerges, the algebraic ``lump''
soliton
\begin{align}
  \label{LumpSigma}
  R = L_\sigma \equiv \left(\frac{4 (\sigma+1)}{1+4\sigma^2
  \xi^2}\right)^\frac{1}{2\sigma}.
\end{align}
Both types of solitons play roles in our study of the blowup profile.

\subsection{Properties of the blow-up profile}

We recall properties of solutions to the profile equation
\eqref{QProfile}.  Details of the proofs can be found in
\cite{Liu:2013ej}.
\begin{proposition}
  \label{prop:ZeroEnergy} Let $Q$ be a classical bounded solution of
  \eqref{QProfile} with $a>0$, such that $Q_\xi \in L^2$ and
  $Q \in L^{4\sigma+2}$. Then its energy and momentum vanish:
  \begin{subequations}
    \begin{align}
      \label{ZeroEnergyQ}
      H(Q) &\equiv \int_\mathbb{R} \left(|Q_\xi|^2 + \tfrac{1}{\sigma+1}|Q|^{2\sigma}
             \Im (\bar{Q} Q_\xi)\right)d\xi = 0\\
      \label{ZeroMomentumQ}
      I(Q) &\equiv \Im \int_\mathbb{R} \bar{Q} Q_\xi d\xi =0.
    \end{align}
  \end{subequations}
\end{proposition}

\begin{proof}
  We multiply \eqref{QProfile} by $\bar{Q}_{\xi\xi}$ and integrate the
  imaginary part of the equation to get
  \begin{align*}
    -a \left(\tfrac{ \sigma+1}{2\sigma}\right) \int
    |Q_\xi|^2 + \int |Q|^{2\sigma} \Re(Q_{\xi\xi}
    \bar{Q}_\xi) = 0.
  \end{align*}
  In the second term we replace $Q_{\xi\xi}$ using \eqref{QProfile}
  leading to
  \begin{align*}
    \int  |Q|^{2\sigma}\Re (Q_{\xi\xi} \bar{Q}_\xi) =
    -\frac{ a}{2\sigma} \int |Q|^{2\sigma} \Im (\bar{Q}
    Q_\xi).
  \end{align*}
  If $a \neq 0$ the identity \eqref{ZeroEnergyQ} follows.
  
  Similarly, multiplying \eqref{QProfile} by $\bar{Q}_\xi$ and taking
  the real part of the equation gives
  \begin{align*}
    \partial_\xi |Q_\xi|^2 + \partial_\xi |Q|^2 + \tfrac{a}{\sigma} \Im \left(\bar{Q}Q_\xi\right) = 0.
  \end{align*}
  If $a >0$, integrating over the real line gives
  \eqref{ZeroMomentumQ}.
\end{proof}

\begin{proposition}
  If $Q$ is a solution of \eqref{QProfile} with $a>0$ and $\sigma >1$,
  and $Q \in H^1 \bigcap L^{2\sigma+2}$, then $Q \equiv 0$.
\end{proposition}
Consequently, there are no nontrivial solutions that belong to
$H^1 \cap L^{2\sigma+2}$.  The behaviour of solutions to
\eqref{QProfile} as $\xi \rightarrow \pm \infty$ can be written as
$Q = A_{\pm} Q_1 + B_{\pm} Q_2$ where $Q_1$ and $Q_2$ behave at
leading order as
$Q_1 \approx |\xi|^{-\frac{ 1}{2\sigma} - \frac{ i}{a}}$ and
$Q_2 \approx e^{-i \frac{a \xi^2}{2}} |\xi|^{1-\frac{1}{2\sigma} +
  \frac{i}{a}}$.
Note that for $\sigma > 1$, $ a>0$,  $Q_1 \notin L^2$ and
$Q_{2_\xi} \notin L^2$.  We are interested in solutions of
\eqref{QProfile} with $B_{\pm} =0$, {\it i.e.} those that behave like
$Q_1$ as $|\xi|\rightarrow \infty$. These are the types of profiles
which correspond to finite energy solutions to the gDNLS equation.
\begin{proposition}
  \label{prop:Asymptotics}
  The large $\xi$ behaviour of zero-energy solutions to
  \eqref{QProfile} is
  \begin{align}\label{eqn:Asymptotics2ndOrder}
    Q = A_{\pm} Q_1 \approx A_\pm |\xi|^{-\frac{ 1}{2\sigma}} \left(1 \pm \frac{b}{2a\sigma |\xi|}\right) e^{-\frac{ i}{a} \left(\ln |\xi| \pm \frac{b}{a|\xi|}\right)}, \quad \xi \to \pm \infty.
  \end{align}
\end{proposition}
This result is a slight refinement of Proposition 4.1 of
\cite{Liu:2013ej}.  A proof is presented in Appendix
\ref{App:Asymptotics}.

\subsection{Phase--Amplitude Decomposition}

To further analyse the profile equation, we introduce the function $P$
defined by
\begin{equation}
  \label{e:QP_relation}
  Q = P \exp\left\{-i\left(\frac{a\xi^2}{4} - \frac{b\xi}{2} +
      \frac{1}{2\sigma+2} \int_{0}^{\xi} |Q|^{2\sigma}\right)\right\}.
\end{equation}
We have extracted a portion of the phase corresponding to the
gDNLS soliton as well as a quadratic part as is often the case in the
study of NLS equations.  $P$ is complex valued and solves
\begin{equation}
  \label{PEquation}
  \begin{split}
    &P_{\xi\xi} + \left(\tfrac{1}{4}(a\xi-b)^2 -1\right) P - i
    \tfrac{a
      (\sigma-1)}{2\sigma}P + \tfrac{1}{2}(a\xi - b) |P|^{2\sigma} P \\
    &\quad+ \tfrac{2\sigma+1}{(2\sigma+2)^2} |P|^{4\sigma} P - \tfrac{
      \sigma}{\sigma+1} |P|^{2(\sigma-1)} \Im(\bar{P} P_\xi) P =0.
  \end{split}
\end{equation}
When $a=0$, we can also assume that $P$ is real valued, and
\eqref{PEquation} becomes \eqref{SolitonEqn}. This illustrates the
connection between the blowup profile and the soliton.  When
$a\neq 0$, the function $P$ is complex valued and it is useful to
decompose it into a real valued amplitude and phase.  Setting
$P = A e^{i\phi}$, we observe that only the derivative of the phase
appears in the equations. Therefore, letting $\psi = \phi_\xi$, we
have the system:
\begin{subequations}
  \begin{gather}
    \label{eqn:Amp_PhaseAmp}
    \begin{split}
      A_{\xi\xi}& + \left(\tfrac{1}{4} (a\xi -b)^2 - 1 -
        \psi^2\right)A + \left( \tfrac{1}{2} (a\xi - b) - \tfrac{
          \sigma}{\sigma+1}
        \psi\right)A^{2\sigma +1} \\
      &\qquad + \tfrac{2\sigma+1}{(2\sigma+2)^2}A^{4\sigma+1} = 0,
    \end{split}\\
    \label{eqn:Psi_PhaseAmp}
    \psi_\xi A + 2\psi A_{\xi} = \tfrac{a(\sigma-1)}{2\sigma} A.
  \end{gather}
\end{subequations}
\eqref{eqn:Psi_PhaseAmp} can be written as
\[
\left(A^2 \psi\right)_\xi = \tfrac{a(\sigma-1)}{2\sigma}A^2
\]
leading to an expression of $\psi$ in terms of $A^2$
\begin{align}\label{eqn:Psi(A)}
  \psi(\xi) = \frac{\psi(0)A^2(0)}{ A^2(\xi)} +
  \frac{a(\sigma-1)}{2\sigma A^2(\xi)} \int_{0}^{\xi}
  A^2(\eta)d\eta.
\end{align}
Alternatively, writing \eqref{eqn:Psi_PhaseAmp} as
\[
\frac{\left(A^2 \psi\right)_\xi}{A^2 \psi} =
\frac{a(\sigma-1)}{2\sigma \psi},
\]
we have the relation
\begin{align}\label{eqn:A^2(psi)}
  A^2(\xi) = \frac{C^2}{|\psi|} \exp\set{\frac{a(\sigma-1)}{2\sigma}
  \int_{\xi_0}^{\xi} \frac{d\eta}{\psi(\eta)}}.
\end{align}
$C^2 = A^2(\xi_0) |\psi(\xi_0)|$ is a constant of integration. In both
cases, the unknown constants depend on $\sigma$.  For reference, the
derivatives of the phase of $Q$ and that of $P$ are related as
\begin{equation}
  \label{e:phase_relation}
  \theta_\xi = \psi - \frac{a\xi-b}{2} - \frac{1}{2\sigma+2}|Q|^{2\sigma}
\end{equation}

\section{Numerical simulation of the profile equation}
  
Here, we briefly summarize our approach to solving \eqref{QProfile}
and make some preliminary obsevations on the profiles.

\subsection{Solvability and Boundary Conditions}

To solve for the profile $Q$, and the parameters $a$ and $b$, it is
necessary to impose a sufficient number of boundary conditions and
solvability conditions.  These are as follows : (i) \ Since the
profile equation is invariant under multiplication by a constant
phase, we assume $Q(0) \in \mathbb{R}$; (ii) Proposition
\ref{prop:Asymptotics} gives a far-field, asymptotically linear
approximation of $Q$ which will be used to construct Robin boundary
conditions, eliminating the constants $A_+$ and $A_-$; (iii) The
parameter $b$ can be changed by a translation in $\xi$.  Under the
assumption that the profile has a unimodal amplitude, which is
consistent with numerical observations, we assume that the maximum of
the amplitude occurs at the origin.  We express this internal boundary
condition as $|Q|_\xi (0) = 0$.

Preliminary numerical simulations of \eqref{QProfile} subject to above
boundary conditions were performed in \cite{Liu:2013ej}. It was
observed that the amplitude $A$ of $Q$ is highly asymmetric and that
the parameter $a$ tends rapidly to zero as $\sigma \rightarrow 1$. In
the next section, we improve these numerical results and make
observations that will guide us in the asymptotic analysis near
$\sigma = 1$. In particular, we identify regions of validity of
different approximations and corresponding turning points.

\subsection{Numerical methods}

The blow-up profile is computed for a sequence of values of $\sigma$
approaching one by continuation. We use a second-order finite
difference scheme, together with a Newton solver to solve the system
for a given value of $\sigma$. Each successful computation is used as
a starting guess for the next smaller value of $\sigma$ in the
sequence. We also make use of Richardson extrapolation to improve upon
computed quantities, such as the parameters $a$ and $b$.  We computed
the solution over a range of $\sigma$ from $\sigma=2$ down to $1.044$,
below which our solver struggled.  We report quantities computed from
this interval, $\sigma \in [1.044, 2]$.  Details on
the numerical methods can be found in Appendix \ref{a:numerics}.

\subsection{Numerical Observations}\label{sect:Observations}

Figure \ref{fig:3dLinProfile} shows the amplitude $|Q|$
near the origin for several values of $\sigma $ close to 1,
computed with the above method. As mentioned before, we see that $|Q|$
is highly asymmetric, decaying much faster for $\xi >0$ than for
$\xi <0$.
\begin{figure}[h]
  \centerline{\includegraphics[height = 8cm]{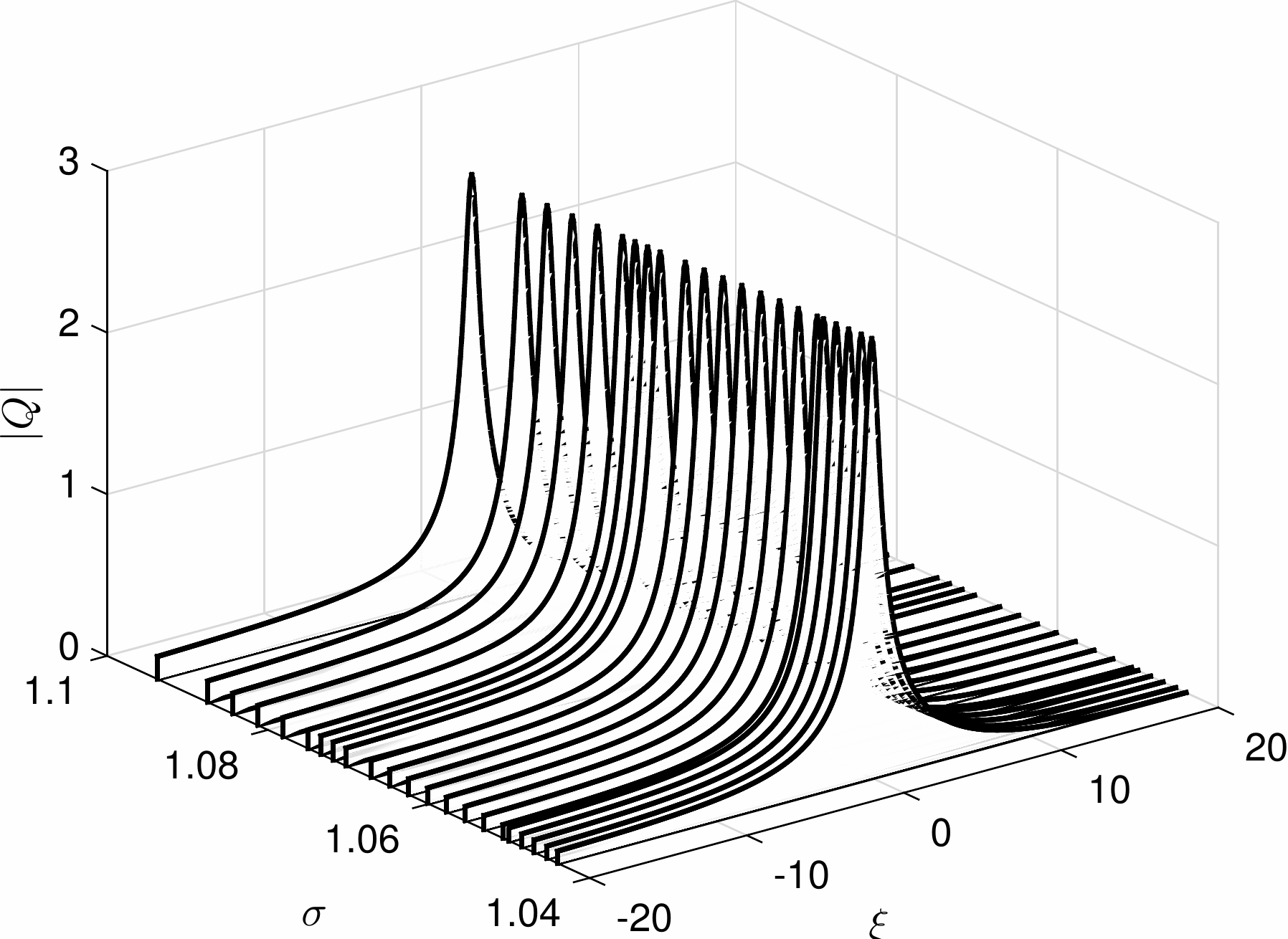}}
  \caption{Amplitude $|Q|$ of the blowup profile for various values of
    $\sigma$ close to 1.}
  \label{fig:3dLinProfile}
\end{figure}

As $\sigma$ decreases, the parameter $a$ decreases rapidly to zero and
the profile $Q$ tends to a soliton solution of the DNLS. The parameter
$b$ increases to a limiting value
$b_0 \equiv\lim\limits_{\sigma \rightarrow 1} b$. Recall that soliton
solutions \eqref{BrightSigma} and \eqref{LumpSigma} to
\eqref{SolitonEqn} are defined for $|b| < 2$ and $b = 2$
respectively. When $\sigma = 1$ and $|b| <2$, the Hamiltonian of
\eqref{e:gdnls_soliton_form} (with $R = B_1$) is
$ H(B_1) = -b \sqrt{4-b^2}$ and its momentum is
$P(B_1) = -2 \sqrt{4-b^2}$, while when $b=2$ (and $R = L_1$) both the
energy and momentum vanish. By construction, the profile $Q$ has a
vanishing Hamiltonian and momentum. We thus make the Ansatz
$\epsilon \equiv 2-b \rightarrow 0$ while $\frac{ \epsilon}{a} \gg 1$.
We will show that this assumption leads to a consistent asymptotic
analysis of all the parameters. The behaviour of $a$ and $\epsilon$
for a range of values of $\sigma$  is illustrated in Figure
\ref{fig:a_eps}.

\begin{figure}
  \centering \subfigure[$a$
  vs. $\sigma$]{\includegraphics[width=6.25cm]{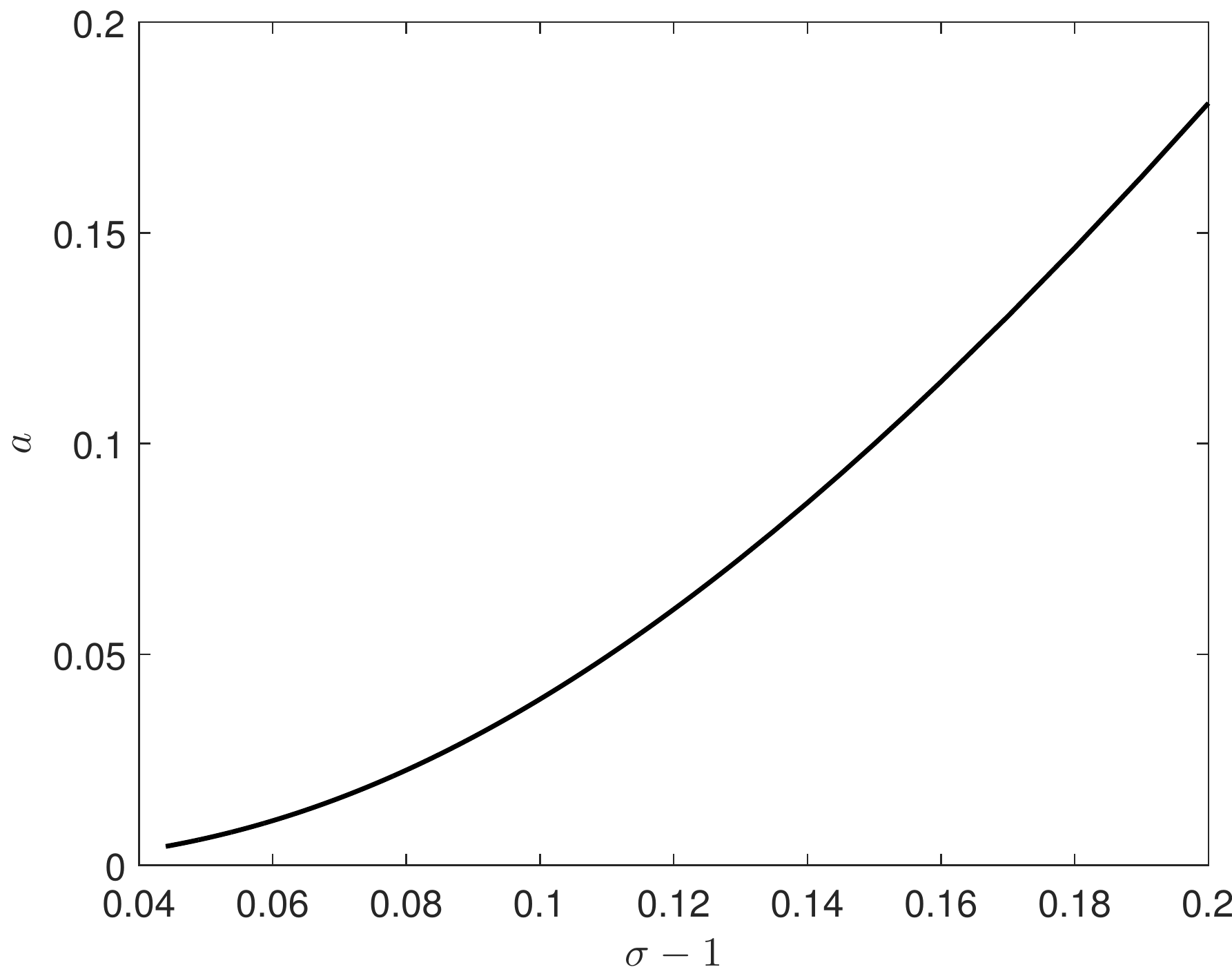}}
  \subfigure[$\epsilon$
  vs. $\sigma$]{\includegraphics[width=6.25cm]{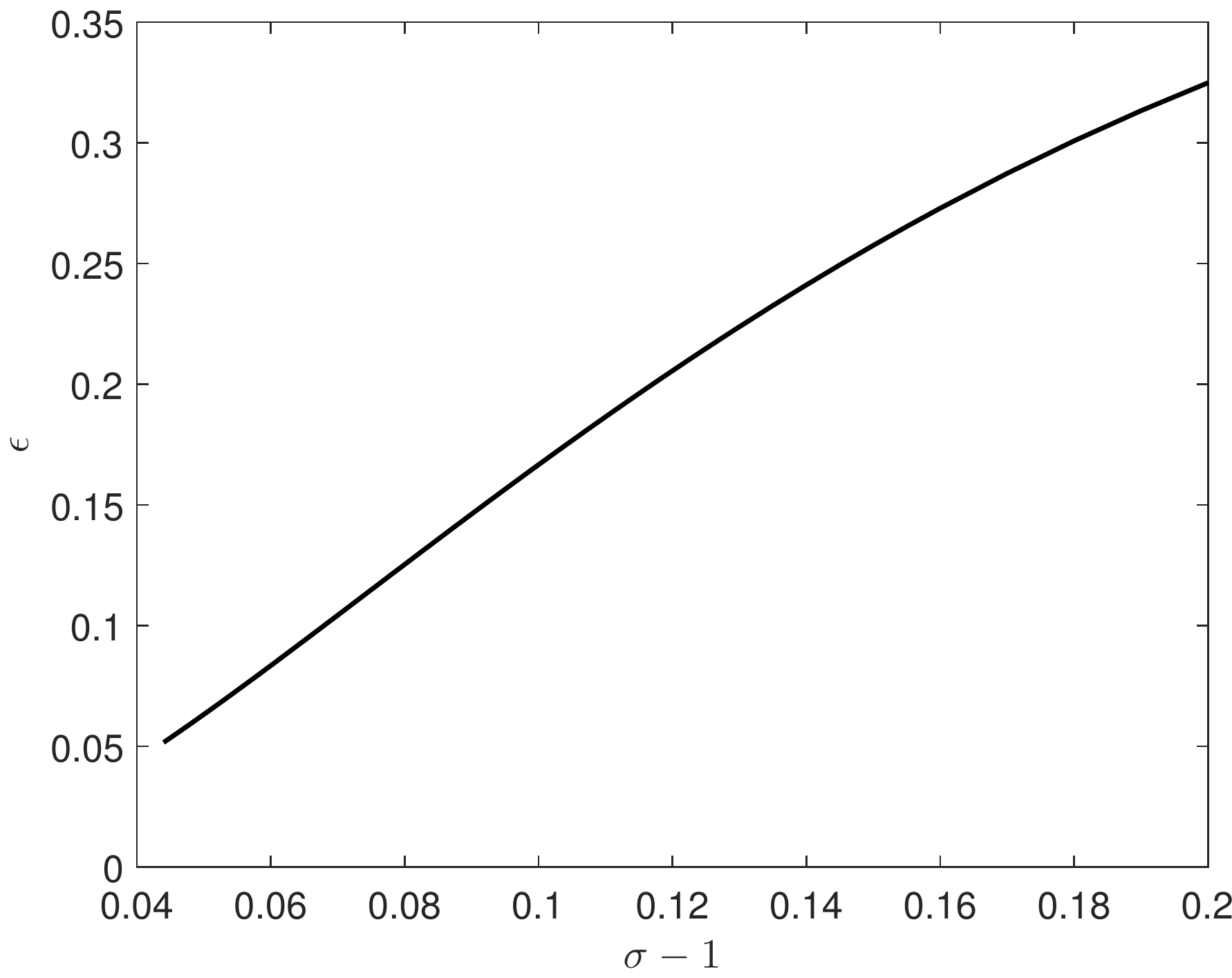}}
  \caption{Numerically computed parameters  $a$ and $\epsilon = 2-b$ for  a range of
    $\sigma \in [1.044, 1.2]$.}
  \label{fig:a_eps}
\end{figure}

Turning to the amplitude and phase equations of $Q$,
\eqref{eqn:Amp_PhaseAmp} and \eqref{eqn:Psi_PhaseAmp}, we see in
Figure \ref{fig:PsiMultiTurning} that $\psi \equiv (\arg P)_\xi$ is
very small in a large region containing the origin and the modified
profile $P$ is essentially real. Rewriting \eqref{eqn:Amp_PhaseAmp} in
terms of $\epsilon$ gives
\begin{equation}\label{eqn:AmpEqnEpsilon}
  \begin{split}
    & A_{\xi\xi} + \left(\tfrac{1}{4} (a\xi + \epsilon)^2 - (a\xi +
      \epsilon) - \psi^2\right)A \cr & \qquad + \left(-1 +
      \tfrac{1}{2} (a\xi + \epsilon) - \tfrac{ \sigma}{\sigma+1}
      \psi\right)A^{2\sigma +1} +
    \tfrac{2\sigma+1}{(2\sigma+2)^2}A^{4\sigma+1} = 0.
  \end{split}
\end{equation}
If $\psi$ is very small, the linear term in this equation reduces to
\begin{align}\label{eqn:linterm}
  \left(\tfrac{1}{4}(a\xi + \epsilon)^2 - (a\xi + \epsilon)\right)A
\end{align}
which is negative if
$\xi \in \left(-\frac{ \epsilon}{a}, \frac{ 4- \epsilon}{a}\right)$.
We thus define the turning points
\begin{equation}
  \label{e:turning}
  \xi_{-} \equiv - \frac{ \epsilon}{a}, \quad \xi_{+} \equiv
  \frac{4-\epsilon}{a}.
\end{equation} 
Figure \ref{fig:PsiMultiTurning} shows how the behaviour of $\psi$
changes near these points for several values of $\sigma$. For large
$|\xi|$, the term \eqref{eqn:linterm} is positive, it must be
compensated by $\psi^2$ to avoid oscillations of the amplitude. As
$|\xi| \rightarrow \infty$, $\psi \approx \frac{a\xi}{2}$ and we
confirm in Figure \ref{fig:PsiMultiTurning} that $\psi$ achieves this
behaviour when $|\xi|$ is much larger than the turning points.

\begin{figure}[h]
  \centerline{\includegraphics[height = 8cm]{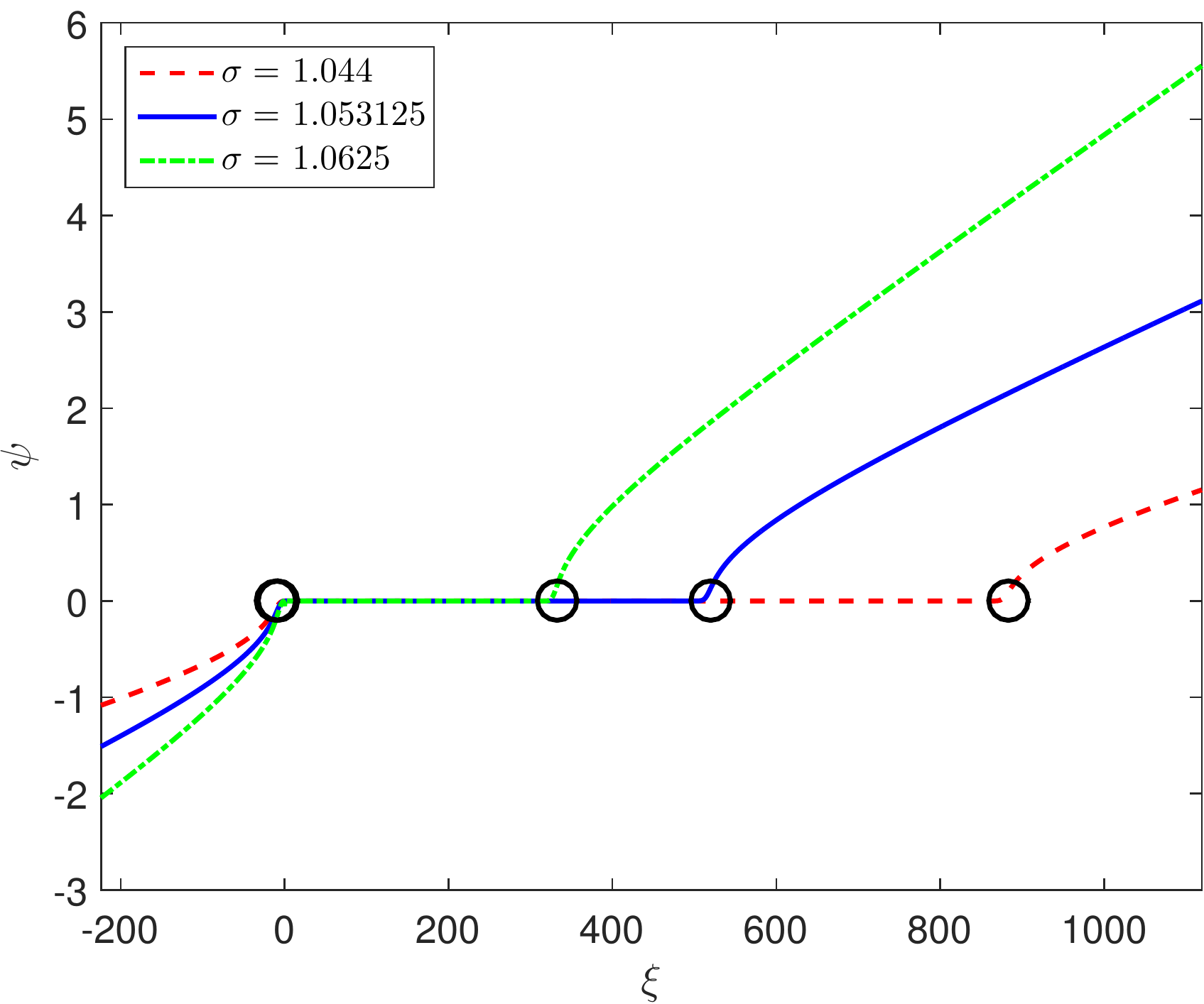}}
  \caption{Phase derivative $\psi$ at several values of $\sigma$. Note
    the change in behaviour near the turning points
    $\xi_{-} = -\frac{ \epsilon}{a}$ and $\xi_{+} = \frac{4}{a}$.}
  \label{fig:PsiMultiTurning}
\end{figure}

\begin{figure}
  
  \subfigure[]{\includegraphics[width=6.25cm]{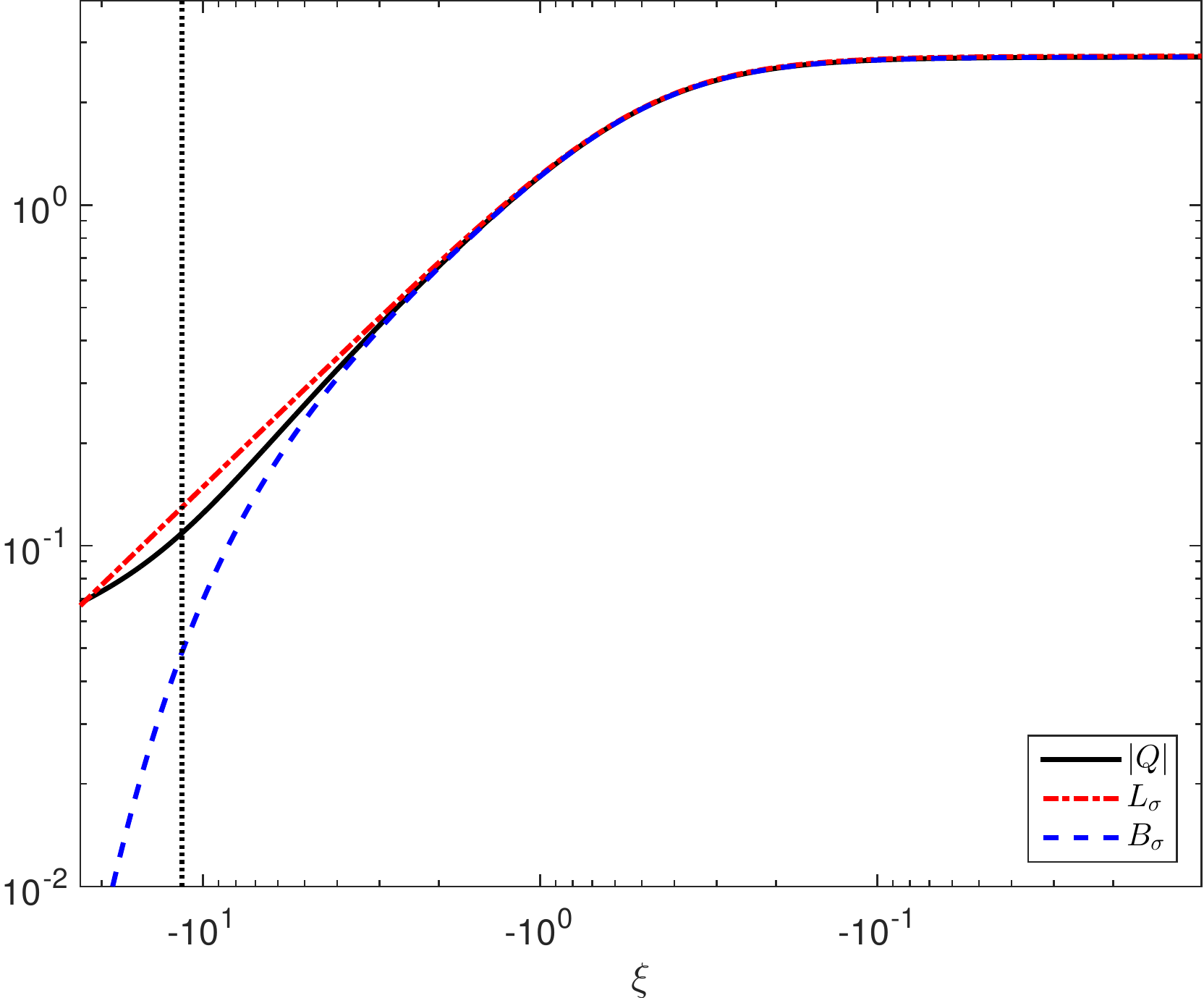}}
  \subfigure[]{\includegraphics[width=6.25cm]{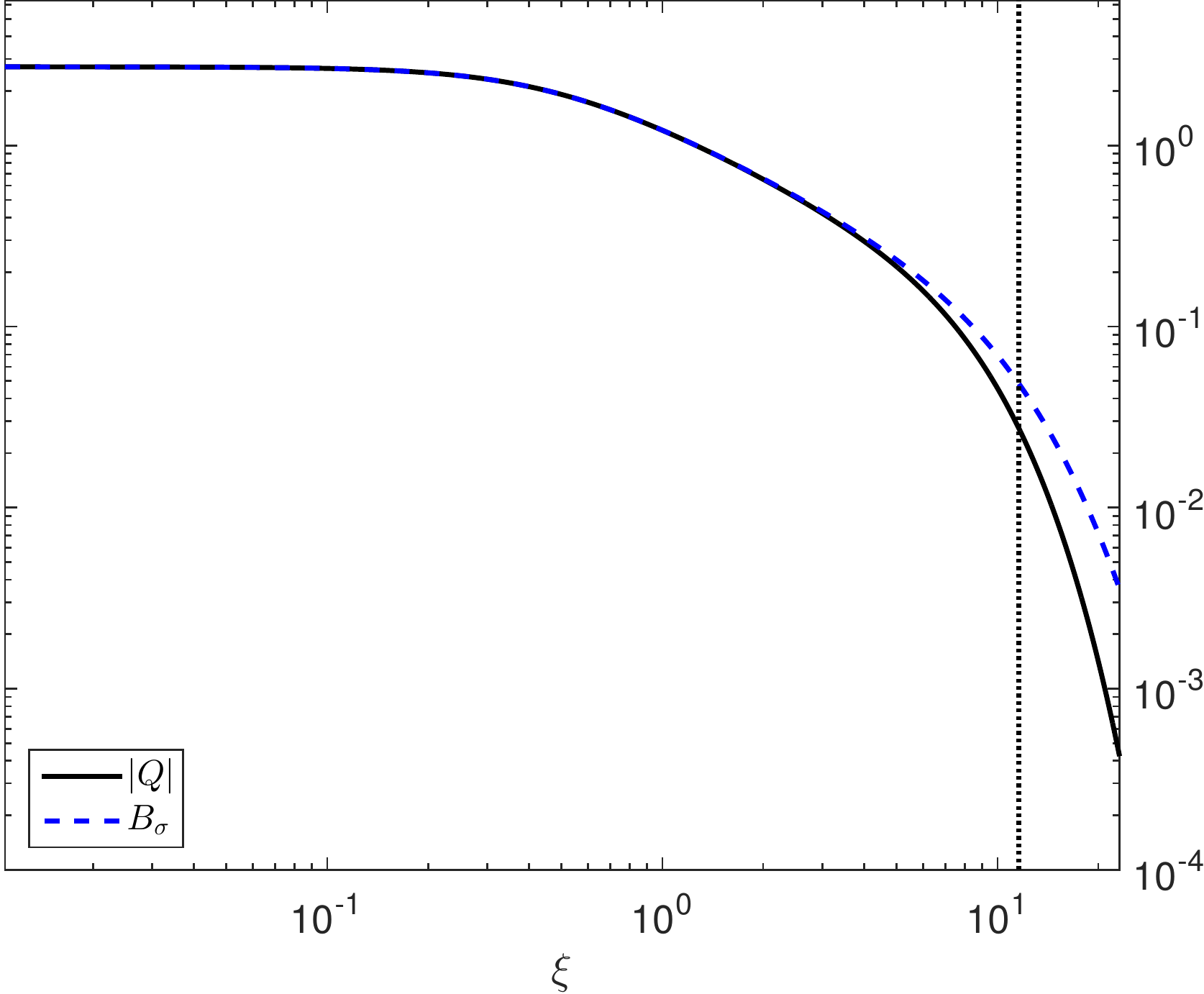}}
  
  \caption{$|Q|$ calculated at $\sigma = 1.044$ and compared to both
    the lump and bright solitons for
    $-\frac{ \epsilon}{a} \leq \xi \leq 0$ (left) and the bright
    soliton for $0 \leq \xi \leq \frac{ \epsilon}{a}$ (right). The
    vertical lines correspond to $|\xi| = \frac{ \eps}{a}$} .
  \label{fig:AmpBrightLump}
\end{figure}

We now consider the amplitude equation \eqref{eqn:AmpEqnEpsilon}. For
$0 \leq \xi \ll \frac{ \epsilon}{a}$,
$a\xi + \epsilon \approx \epsilon$ and \eqref{eqn:AmpEqnEpsilon}
essentially reduces to \eqref{SolitonEqn} satisfied by the bright
soliton \eqref{BrightSigma} with parameter $b = 2- \epsilon$. On the
negative side, for $ \xi \leq 0$, the terms $a\xi$ and $\epsilon$ work
against each other and we find that the lump soliton \eqref{LumpSigma}
better approximates the solution.  However, when $\xi$ approaches
$\xi_- = - \frac{\eps}{a}$, the phase becomes important and the
amplitude deviates from \eqref{LumpSigma} as shown in Figure
\ref{fig:AmpBrightLump} (a). This deviation is a source of difficulty
in the asymptotic analysis.  Figure \ref{fig:AmpBrightLump} (b)
displays $|Q|$ for $\xi > 0 $ compared to the bright soliton
\eqref{BrightSigma}.

\begin{remark}
  For $\xi \ll \epsilon^{-1/2}$, the bright and lump solitons nearly
  coincide.
\end{remark}
We are now in a position to better interpret the asymmetry of the
profile amplitude. The turning point $\xi_+ \approx \frac{4}{a}$ grows
very rapidly. When $1 \ll \xi < \xi_{+}$, the nonlinear terms in
\eqref{eqn:AmpEqnEpsilon} are negligible and the negative linear term
\eqref{eqn:linterm} forces the amplitude to decay very rapidly. Figure
\ref{fig:AmpPlusMultiTurning} shows the amplitude for several values
of $\sigma$ close to $1$ and we clearly see this fast decay up to the
turning point $\xi_{+}$. In this region the WKB method provides a good
approximate solution. Meanwhile, on the negative side,
$|\xi_{-}| = \frac{ \epsilon}{a}$ grows moderately and the linear term
is very small for $\xi \in \left(\xi_{-}, 0\right)$.  We observe
only a moderate decay of the amplitude for negative values of
$\xi$. When $\xi < \xi_-$ away from the turning point, the WKB method
provides a good approximation to the solution.  Finally, when $|\xi|$
is large and far away from the turning points, the amplitude is well
approximated by the leading order asymptotics
$|Q| \approx A_{\pm} |\xi|^{-\frac{ 1}{2\sigma}}$.

In the next section, we derive a formal asymptotic analysis motivated
by these observations and describe the leading order behaviour of the
parameters $a$ and $\epsilon$ as $\sigma \rightarrow 1$.

\begin{figure}[h]
  \centerline{\includegraphics[height = 8cm]{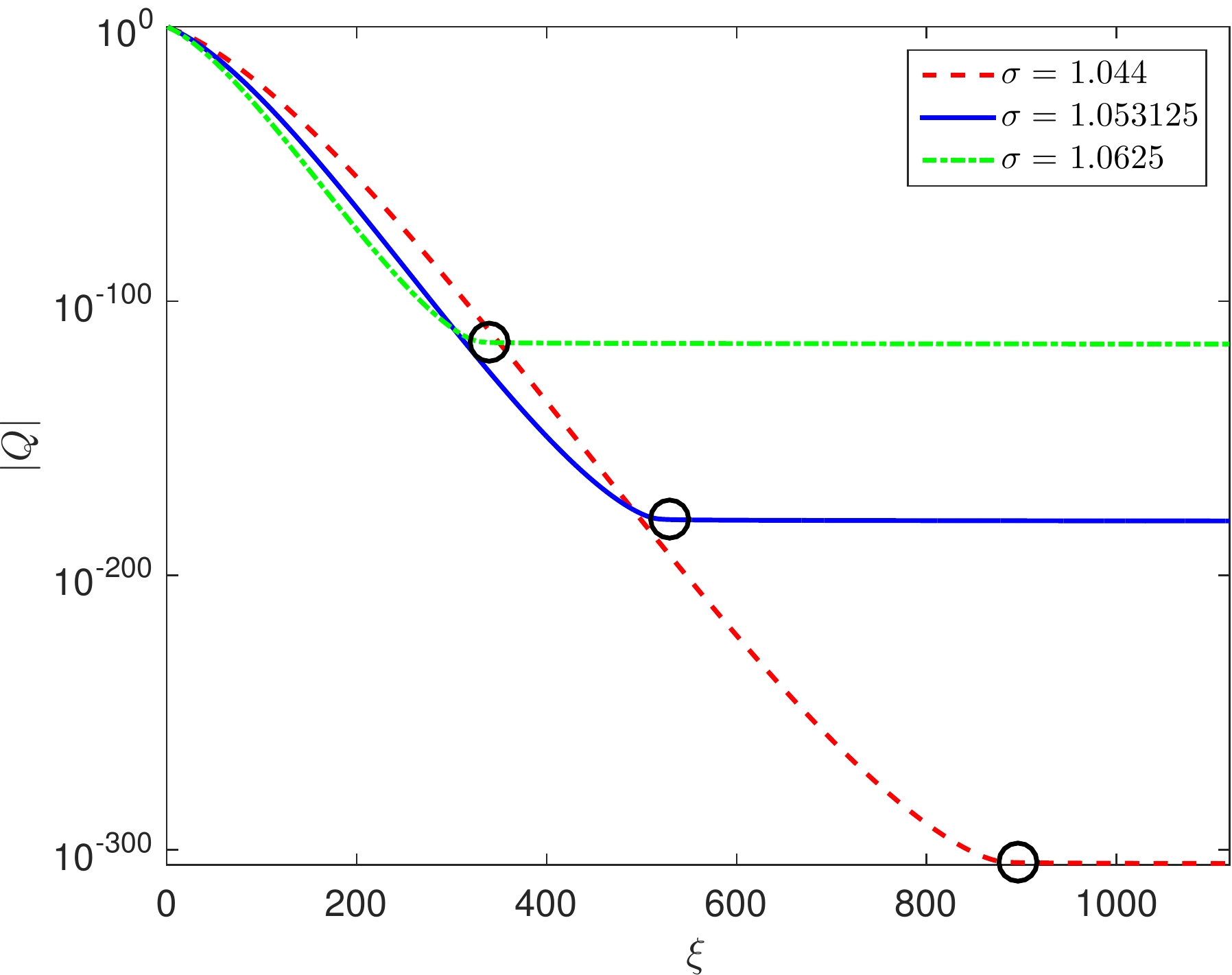}}
  \caption{$|Q|$ computed at several values of $\sigma$. Note the
    rapid decay up to the turning point $\xi_{+} = \frac{4}{a}$ marked
    with $\bigcirc$.}.
  \label{fig:AmpPlusMultiTurning}
\end{figure}

\section{Asymptotic analysis}

The numerics indicate that, as $\sigma \rightarrow 1$, the parameters
$a(\sigma)$, $b(\sigma)$ tend to 0 and 2 respectively, while the
profile $Q$ tends to the lump soliton \eqref{LumpSigma}. In this
section, we investigate the deformation of $Q$ and the parameters
$a(\sigma)$, $b(\sigma)$ using asymptotic methods and analysis of the
numerical data. In the course of the calculation, three additional
parameters come into play, the coefficients $A_{+}$, $A_{-}$ appearing
in the large $|\xi|$ behaviour of $Q$ (see
equation\eqref{eqn:Asymptotics2ndOrder}) and the derivative of the
phase at the origin $\psi(0)$.

Section \ref{sect:xi>0} concentrates on the region $\xi > 0$. We
connect the bright soliton \eqref{BrightSigma} which approximates $P$
(defined in \eqref{e:QP_relation}) close to the origin to the
asymptotic behaviour at large $\xi$ \eqref{eqn:Asymptotics2ndOrder}
using WKB method. We obtain two relations between the above
parameters, given in equations \eqref{eqn:AplusFinalRelation} and
\eqref{eqn:psi0AsymptoticForm}.

In Section \ref{sect:xi<0} we examine the region $\xi < 0$.  Close to
the origin, the lump soliton \eqref{LumpSigma} approximation is valid,
however nonlinear effects become important near the turning point
$\xi_-$.  A precise analytic form of the profile in the (relatively
small) region containing the turning point remains an open problem,
nevertheless we are able to find a relation between the
parameters (Eq \eqref{eqn:AminAymptoticForm}). Lacking a precise
description of the profile in the intermediate region, we carefully analyze our numerical data
and find that the turning point behaves like a power law in
$(\sigma-1)$: $ a/\eps \sim \paren{\sigma-1}^\alpha $ with
$\alpha \approx 1.2$. (Section \ref{sect:TPNumerics}.)

\subsection{Asymptotic analysis of the profile for $\xi > 0$}
\label{sect:xi>0}

\begin{proposition}\label{prop:APlusRelation}
  As $\sigma \rightarrow 1$, the behaviour of the coefficient $A^+$
  defined in \eqref{eqn:Asymptotics2ndOrder} is given by 
  \begin{align}
    \label{eqn:AplusFinalRelation}
    A_{+} \approx 4 \epsilon^{3/4} a^{-1/2} \exp\left\{-\frac{\pi}{a} + \frac{2}{3} \frac{ \epsilon^{3/2}}{a}\right\}.
  \end{align}
\end{proposition}

\begin{proof}
  We use the approach presented in Chapter 8 of \cite{Sulem:1999kx} to
  connect the behaviour of $Q$ as $\xi \to +\infty$ to the bright
  soliton approximation valid for $\xi \ll \frac{ \epsilon}{a}$.
  First, we introduce the function $S$, which relates to $Q$ by
  \begin{equation}
    \label{e:SQ_relation}
    S =Q\exp\left\{i \left(\frac{a\xi^2}{4} - \frac{b\xi}{2}\right)\right\}.
  \end{equation}
  $S$ satisifies
  \begin{align}\label{eqn:SFull}
    S_{\xi\xi} - \left(1- \tfrac{1}{4} (a\xi - b)^2\right) S - i a \tfrac{ \sigma-1}{2\sigma} S + \tfrac{1}{2} (a\xi-b) |S|^{2\sigma} S + i |S|^{2\sigma} S_\xi =0,
  \end{align}
  and as $\xi \rightarrow + \infty$,
  \begin{align}\label{eqn:SAsymp}
    S_{\Asymp} = A_{+} \xi^\frac{ -1}{2\sigma} \exp\left\{i\left(\frac{a\xi^2}{4} - \frac{b\xi}{2} - \frac{1}{a} \ln \xi\right)\right\}.
  \end{align}
  For sufficiently large $\xi \gg 1$, the nonlinear terms in
  \eqref{eqn:SFull} are negligible and we may write
  \begin{align}\label{eqn:S_xigg1}
    S_{\xi\xi} = \left(1- \tfrac{1}{4} (a\xi - b)^2\right)S.
  \end{align}
  Setting $x = \frac{1}{2} ( a\xi - b)$, \eqref{eqn:S_xigg1} becomes
  \begin{align}\label{eqn:S_x}
    \tfrac{a^2}{4} S_{xx} = (1-x^2)S,
  \end{align}
  and the solution can be approximateed by the WKB method. For $x>1$,
  \begin{align}\label{eqn:SWKBR}
    S_{\WKB}^R = \frac{1}{(x^2 - 1)^\frac{1}{4}} \left(C^R e^{i\left(\frac{ \pi}{4} + \int_{1}^{x} \sqrt{s^2 - 1}ds\right)} + D^R e^{i\left(\frac{ \pi}{4} - \int_{1}^{x} \sqrt{s^2 - 1}ds\right)}\right) .
  \end{align}
  When $x \gg 1$,
  \begin{align}
    \frac{2}{a} \int_{1}^{x} \sqrt{s^2 - 1}ds \approx \frac{1}{a} (x^2 - \ln x) \approx \frac{a\xi^2}{4} - \frac{b\xi}{2} - \frac{1}{a} \ln \xi,
  \end{align}
  implying that $D^R = 0$. Matching the amplitudes of
  \eqref{eqn:SWKBR} and \eqref{eqn:SAsymp}, we find
  \begin{align}\label{rel:CR_Aplus}
    C^R = \sqrt{\frac{a}{2}} A_{+}.
  \end{align}
  The right hand side of \eqref{eqn:S_x} vanishes at the turning point
  $x = 1$.  The WKB approximation \eqref{eqn:SWKBR} is valid for
  $x - 1 \gg a^\frac{2}{3}$. If, in addition, $x - 1 \ll 1$,
  \eqref{eqn:SWKBR} can be simplified to
  \begin{align}\label{eqn:SWKBR_xnear1}
    S_{\WKB}^R \approx \frac{C^R}{(2(1-x))^\frac{1}{4}} e^{i \left(\frac{ \pi}{4} + \frac{4 \sqrt{2}}{3a} (x-1)^\frac{3}{2}\right)}.
  \end{align}
  On the other hand, when $|x - 1| \ll 1$, we replace \eqref{eqn:S_x}
  by
  \begin{align}\label{eqn:S_xnear1}
    \tfrac{a^2}{4} S_{xx} = 2(1-x) S.
  \end{align}
  In the variable $t = 2 a^{-\frac{ 2}{3}} (1-x)$,
  \eqref{eqn:S_xnear1} is the Airy equation
  \begin{align}\label{eqn:AiryEqn}
    S_{tt} = t S
  \end{align}
  whose solution is $S_{\Airy} = a_1 \Ai(t) + a_2 \Bi (t)$. In terms
  of the variable $t$, the region $a^\frac{2}{3} \ll x-1 \ll 1$
  corresponds to $(-t) \gg 1$. Using the asymptotics of $\Ai$ and $Bi$
  as $t \rightarrow -\infty$,we obtain
  \begin{align}\label{eqn:SAiry_tnegInf}
    S_{\Airy} \approx \frac{1}{\sqrt{\pi} (-t)^\frac{1}{4}} \left(a_1
    \sin \left(\frac{ \pi}{4} + \frac{2}{3}(-t)^\frac{3}{2}\right) +
    a_2 \cos \left(\frac{ \pi}{4} +
    \frac{2}{3}(-t)^\frac{3}{2}\right)\right). 
  \end{align}
  Matching the phases of \eqref{eqn:SWKBR_xnear1} and
  \eqref{eqn:SAiry_tnegInf} requires $a_1 = i a_2$, and matching the
  amplitudes gives
  \begin{align}\label{rel:a2_CR}
    a_2 = a^{-\frac{ 1}{6}} \sqrt{\pi} C^R .
  \end{align}
  For $x<1$, the region $1 \gg 1-x \gg a^\frac{2}{3}$ corresponds to
  $t \gg 1$. Using the asymptotics of $\Ai$ and $\Bi$ as
  $t \rightarrow + \infty$, we obtain
  \begin{align}\label{eqn:SAiry_tposInf}
    S_{\Airy} \approx \frac{1}{\sqrt{\pi} t^\frac{1}{4}} \left( \frac{a_1}{2} e^{ \frac{ -2}{3} t^\frac{3}{2}} + a_2 e^{\frac{2}{3} t^\frac{3}{2}}\right) \approx \frac{a_2}{\sqrt{\pi} t^\frac{1}{4}} e^{\frac{2}{3} t^\frac{3}{2}}
  \end{align}
  since the term with a negative exponent is negligible. On the other
  hand, solving \eqref{eqn:S_x} for $x<1$ by WKB gives
  \begin{align}\label{eqn:SWKBL}
    S_{\WKB}^L = \frac{1}{(1-x^2)^\frac{1}{4}} \left(C_1^L e^{\frac{2}{a} \int_{1}^{x}\sqrt{1-s^2}ds} + C_2^L e^{- \frac{2}{a} \int_{1}^{x} \sqrt{1-s^2}ds}\right) .
  \end{align}
  Noting that for $(1 - x) \ll 1$
  \begin{align}
    \frac{2}{a} \int_{1}^{x} \sqrt{1-s^2}ds \approx -\frac{4 \sqrt{2}}{3a} (1-x)^\frac{3}{2} ,
  \end{align}
  we have that for $1 \gg 1-x \gg a^\frac{2}{3}$ \eqref{eqn:SWKBL}
  simplifies to
  \begin{align}\label{eqn:SWKBL_near1}
    S_{\WKB}^L \approx \frac{1}{(2(1-x))^\frac{1}{4}} \left(C_1^L e^{- \frac{4 \sqrt{2}}{3a} (1-x)^\frac{3}{2}} + C_2^L e^{\frac{4 \sqrt{2}}{3a} (1-x)^\frac{3}{2}}\right) . 
  \end{align}
  Finally, matching \eqref{eqn:SWKBL_near1} to
  \eqref{eqn:SAiry_tposInf} and solving for $C^L \equiv C_2^L$ gives
  \begin{align}\label{rel:a2_CL}
    a_2 = a^{-\frac{1}{6}} \sqrt{\pi} C^L,
  \end{align}
  where we have again ignored the term with a negative exponent. The
  approximation \eqref{eqn:SWKBL} for $S$ is real valued and we
  connect it to the bright soliton approximation valid for
  $\xi \ll \frac{\epsilon}{a}$ as described in Section
  \ref{sect:Observations}.

  We assume here that $\epsilon \gg a^\frac{2}{3}$. This Ansatz will
  be checked a posteriori.  The WKB approximation thus remains valid
  in some region included in $\xi < \frac{ \epsilon}{a}$. Indeed, we
  work within the region
  $a^{-\frac{ 1}{3} }\ll \xi \ll \frac{ \epsilon}{a}$, equivalently
  $a^\frac{2}{3} \ll x+1 \ll \epsilon$. This condition also ensures
  that $\epsilon^{-\frac{ 1}{2}} \ll a^{-\frac{ 1}{3}}$ and the bright
  soliton can be approximated as
  \begin{align}\label{eqn:BrightLargeXi}
    B_\sigma(\xi) \approx 2\sqrt{2 \epsilon} e^{ -\sqrt{\epsilon}\xi} .
  \end{align}
  In this region $1 - x^2 \approx 2(1+x) = a\xi + \epsilon$, so we
  approximate
  \begin{align}
    \int_{1}^{x} \sqrt{1-s^2}ds \approx -\frac{ \pi}{2} + \frac{2\sqrt{2}}{3} (1+x)^\frac{3}{2} \approx -\frac{ \pi}{2} + \frac{1}{3} \epsilon^\frac{3}{2} + \frac{1}{2} \sqrt{\epsilon} a\xi
  \end{align}
  and the WKB approximation \eqref{eqn:SWKBL} can be written as
  \begin{align}\label{eqn:SWKBL_xnear-1}
    S_{WKB}^L \approx C^L \epsilon^{-\frac{ 1}{4}} \exp\left\{\frac{ \pi}{a} - \frac{2}{3} \frac{ \epsilon^\frac{3}{2}}{a}\right\} e^{-\sqrt{\epsilon} \xi} .
  \end{align}
  Matching \eqref{eqn:BrightLargeXi} and \eqref{eqn:SWKBL_xnear-1}
  gives
  \begin{align}\label{rel:CL_Bright}
    C^L = 2\sqrt{2} \epsilon^\frac{3}{4} \exp\left\{-\frac{ \pi}{a} + \frac{2}{3} \frac{ \epsilon^\frac{3}{2}}{a}\right\} .
  \end{align}
  Finally, combining relations \eqref{rel:CR_Aplus},
  \eqref{rel:a2_CR}, \eqref{rel:a2_CL} and \eqref{rel:CL_Bright}, we
  obtain relation \eqref{eqn:AplusFinalRelation} between $A_{+}$, $a$,
  and $\epsilon$.

\end{proof}

\begin{figure}[h]
  \centerline{\includegraphics[height=8cm]{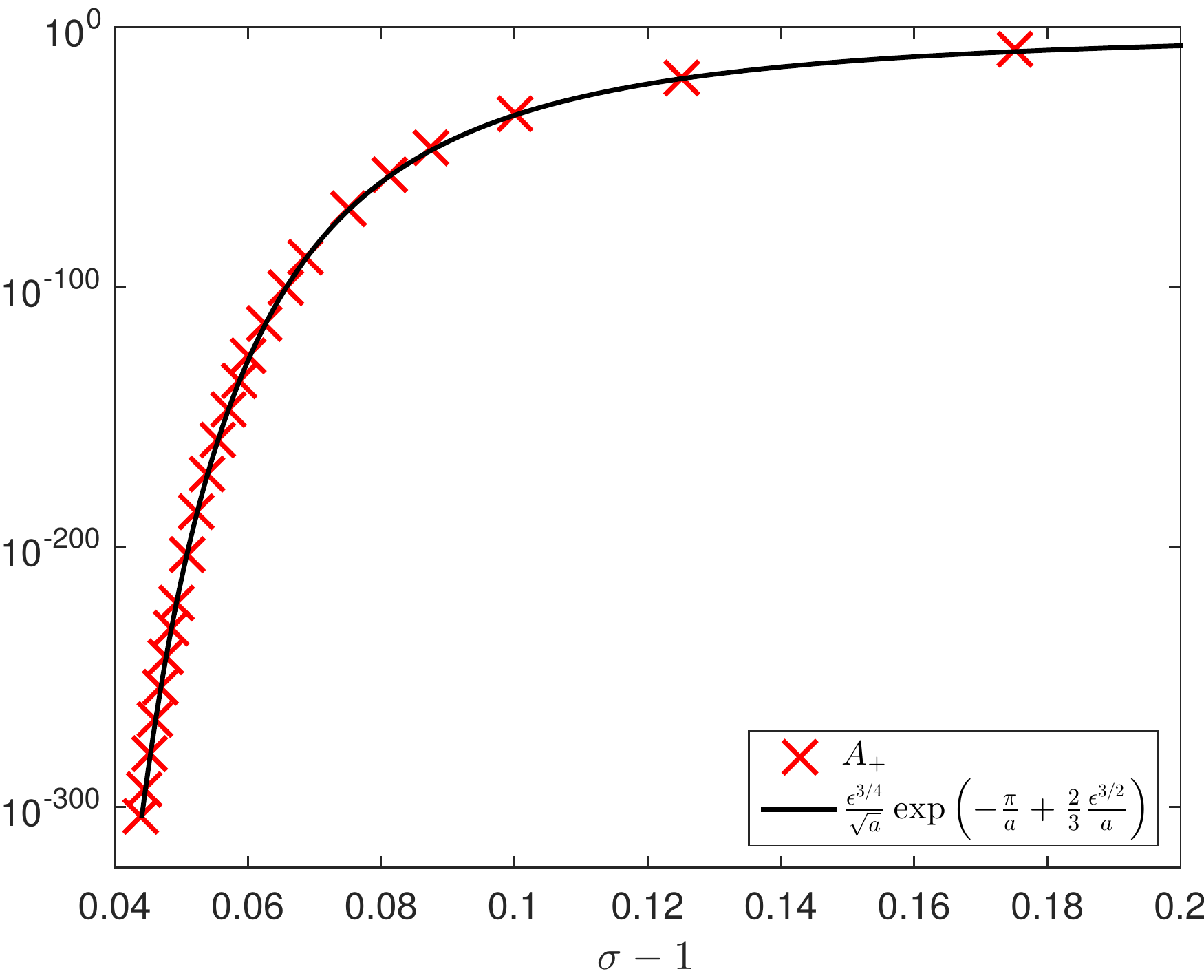}}
  \caption{Numerical verification of \eqref{eqn:AplusFinalRelation}
    relating the coefficient $A_{+}$ to $a$ and $\epsilon$.}
  \label{fig:AplusRelation}
\end{figure}
In Figure \ref{fig:AplusRelation} we verify the relation
\eqref{eqn:AplusFinalRelation} against the value of $A_+$ extracted
from the numerical integration of the boundary value problem and find
an excellent agreement for a large range of values $\sigma$ from
$\sigma = 1.2$ up to the limit of our computation at $\sigma = 1.044$.

\begin{proposition}\label{prop:psi(0)}
  To leading order in $\sigma$ as $\sigma \rightarrow 1$, the
  derivative of the phase at the origin, $\psi(0)$, is given by
  \begin{align}\label{eqn:psi0AsymptoticForm}
    \psi(0) \approx - \tfrac{ \pi a }{8} \left(\sigma - 1\right) .
  \end{align}
\end{proposition}

\begin{proof}
  We turn to the relation \eqref{eqn:Psi(A)} between the phase
  derivative $\psi$ and the amplitude $A$. Take $\xi_0 > \frac{4}{a}$
  sufficiently large so that
  $|Q (\xi_0)| \approx A_{+} \xi^\frac{ -1}{2\sigma}$,
  $\psi(\xi_0) \approx \frac{a \xi_0}{2}$, and denote
  $k \equiv \int_{0}^{\xi_0} A^2$. For $\xi > \xi_0$ we approximate
  \eqref{eqn:Psi(A)} by
  \begin{equation*}
    \begin{split}
      \psi(\xi) &\approx \frac{ \psi(0) A^2(0)}{A_{+}^2} \xi^\frac{1}{\sigma} + \frac{a(\sigma-1)k}{2\sigma A_{+}^2} \xi^\frac{1}{\sigma} + \frac{a (\sigma-1)}{2\sigma A_{+}^2} \xi^\frac{1}{\sigma} \int_{\xi_0}^{\xi} A_{+}^2 \eta^\frac{-1}{\sigma}d\eta \\
      &= \left(\frac{1}{A_{+}^{2}}\left(\psi(0) A^2(0) +
          \frac{a(\sigma-1)k}{2\sigma}\right) - \frac{a}{2}
        \xi_0^{1-\frac{1}{\sigma}}\right) + \frac{a\xi}{2}.
    \end{split}
  \end{equation*}
  Since $A_{+}$ decays exponentially fast, we have
  \begin{align}\label{eqn:psi0Form_with_k}
    \psi(0) \approx - \frac{a(\sigma-1) k }{2\sigma A^2(0)} .
  \end{align}
  The main contribution to the integral $k$ comes from the region
  where the amplitude is approximated by the soliton, therefore to
  leading order as $\sigma \rightarrow 1 $ we have
  $k \approx \int_{0}^{\infty} B_\sigma^2 \approx 2\pi$ and
  $A^2 (0) \approx L_1 (0) = 8$.
\end{proof}
\begin{figure}[h]
  \centerline{\includegraphics[height=8cm]{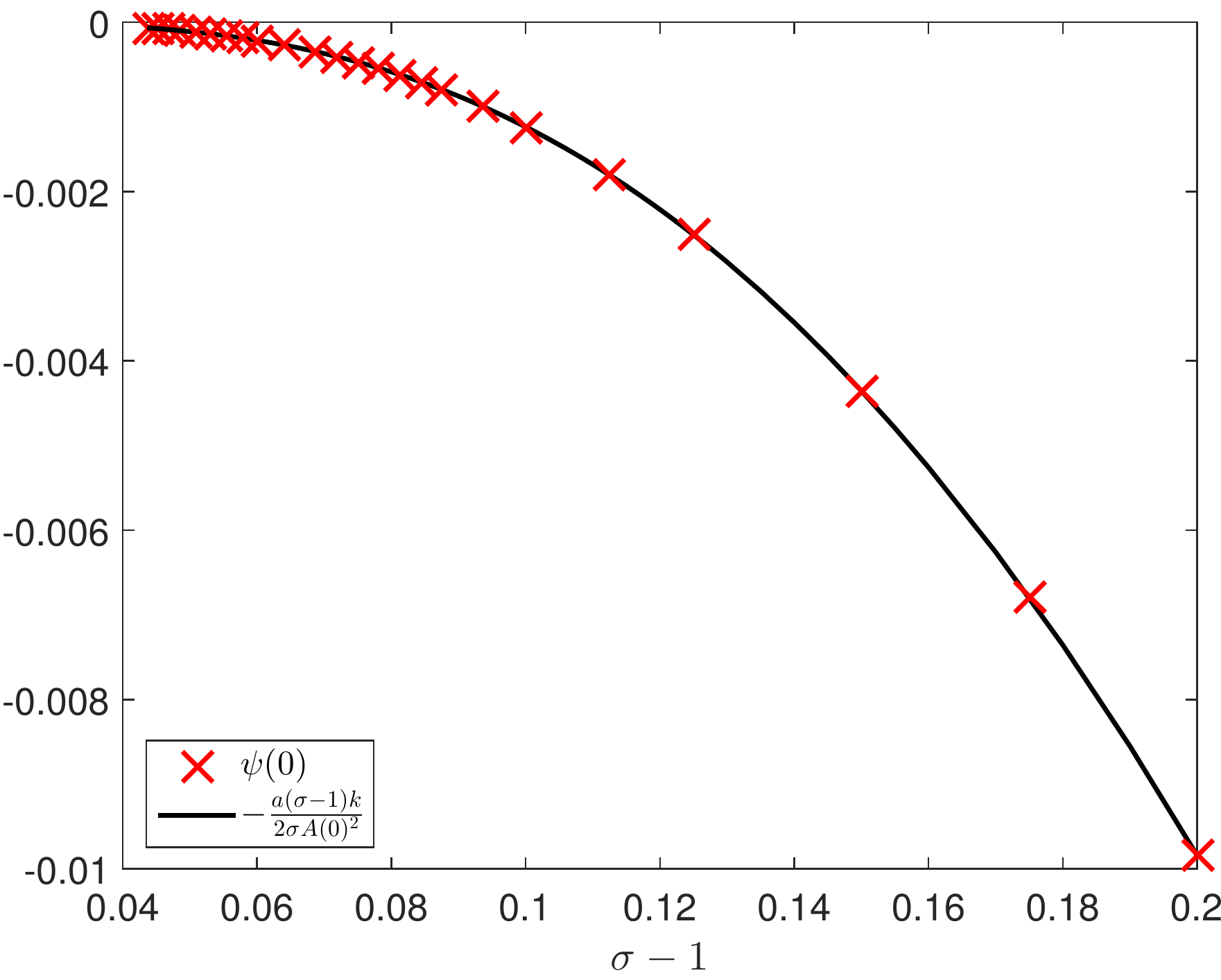}}
  \caption{Numerical verification of relation
    \eqref{eqn:psi0Form_with_k} for $\psi(0)$ for a range of values of
    $\sigma$. }
  \label{fig:psi0Relation}
\end{figure}
Figure \ref{fig:psi0Relation} confirms the relation
\eqref{eqn:psi0Form_with_k} against the numerical simulations, again
finding excellent agreement. We check relation
\eqref{eqn:psi0Form_with_k} rather than \eqref{eqn:psi0AsymptoticForm}
because the values of $\sigma$ at which we compute are insufficiently
close to one for the constant $k$ to have reach its limiting value.

\subsection{Asymptotic analysis of the profile for  $\xi < 0$}\label{sect:xi<0}

\subsubsection{Asymptotics of the parameter $A_-$}

\begin{proposition}\label{prop:Aneg}
  To leading order in $\sigma$ as $\sigma \rightarrow 1$, the
  coefficient $A_{-}$ defined in \eqref{eqn:Asymptotics2ndOrder} is
  given by
  \begin{align}\label{eqn:AminAymptoticForm}
    A_{-} \approx \sqrt{4\pi \left(\sigma-1\right)} .
  \end{align}
\end{proposition}

\begin{proof}
  For sufficiently large $|\xi|$, $\xi<0$, the function $S$ satisfies
  \eqref{eqn:S_xigg1} or, equivalently, defining
  $y = -\frac{ 1}{2}(a\xi - b)$,
  \begin{align}\label{eqn:S_y}
    \tfrac{a^2}{4} S_{yy} = (1-y^2) S .
  \end{align}
  Using WKB, we have for $y-1 \gg a^\frac{2}{3}$ (equivalently
  $|\xi - \frac{ \eps}{a}| \gg a^\frac{ -1}{3}$)
  \begin{align*}
    \arg (S) \approx \frac{ \pi}{4} + \frac{2}{a} \int_{1}^{y} \sqrt{s^2 - 1} ds
  \end{align*}
  from which it follows
  \begin{align}\label{eqn:psiNegSqrt}
    \psi \approx - \sqrt{y^2 - 1}, \text{  and  } A \approx \frac{ \sqrt{\frac{a}{2}} A_{-}}{(y^2-1)^\frac{1}{4}}  .
  \end{align}
  We improve the approximation of the amplitude by using
  \eqref{eqn:A^2(psi)}, giving us
  \begin{align}\label{eqn:A_y<1_approximation}
    A \approx \frac{C_{-}}{(y^2-1)^\frac{1}{4}} \left(y + \sqrt{y^2 -1}\right)^\frac{ \sigma - 1}{2\sigma}.
  \end{align}
  When $\xi \rightarrow - \infty$, $y \rightarrow + \infty$ and
  \eqref{eqn:A_y<1_approximation} becomes
  $A \approx \sqrt{2} C_{-} a^{-\frac{ 1}{2\sigma} } (-\xi)^{-\frac{
      1}{2\sigma}}$.
  Using Proposition \ref{prop:Asymptotics}, we have
  $A \approx A_{-} (-\xi)^{-\frac{ 1}{2\sigma}}$ when
  $\xi \to -\infty$.  Thus the constants $C_{-}$ and $A_{-}$ are
  related by
  \begin{align*}
    C_{-} = \frac{ a^\frac{1}{2\sigma}}{\sqrt{2}} A_{-} .
  \end{align*}
  Returning to equation \eqref{eqn:Psi(A)} for large negative
  $|\xi| \gg \frac{1}{a}$ and approximating $A$ by its
  asymptotic behaviour we write
  \begin{equation}\label{eqn:psi_neg_integrals}
    \begin{split}
      \psi(\xi) \approx &\frac{ \psi(0) A^2(0)}{2\sigma A_{-}^2}
      |\xi|^\frac{1}{\sigma} + \frac{a(\sigma - 1) }{2\sigma A_{-}^2}
      \left(\int_{0}^{-\frac{\epsilon}{a} - a^{-1/3}}
        A^2\right)|\xi|^\frac{1}{\sigma} \cr & - \left(1-
        \frac{1}{\sigma}\right) \frac{C_{-}^2}{A_{-}^2}
      |\xi|^\frac{1}{\sigma} \int_{1 + a^{2/3}}^{\infty} \frac{
        \left(y^\prime + \sqrt{y^{\prime^2} -
            1}\right)}{\sqrt{y^{\prime^2} - 1}} dy^\prime .
    \end{split}
  \end{equation}
  We do not have a precise behaviour of the profile in the relatively
  small region between $-\frac{ \eps}{a}$ and
  $-\frac{\eps}{a} - a^{-\frac{1}{3}}$, but we have numerically
  verified that the contribution of this small region to the above
  integrals is negligible compared to the contribution of the interval
  $(0, -\frac{ \eps}{a})$. Denoting
  $l = \int_{-\frac{\epsilon}{a}}^{0} A^2$ and using the expression
  \eqref{eqn:psi0Form_with_k} for $\psi(0)$, we write
  \begin{equation}
    \begin{split}
      \psi(\xi) &\approx - \frac{a (\sigma-1)(k+l)}{2\sigma A_{-}^2} |\xi|^\frac{1}{\sigma} - \left(1-\tfrac{1}{\sigma}\right) \frac{C^2}{A_{-}^2} \int_{1}^{y} \frac{ \left(y^\prime + \sqrt{y^{\prime^2} - 1}\right)}{\sqrt{y^{\prime^2} - 1}} dy^\prime  \\
      &=\frac{1}{2} \left(a^\frac{1}{\sigma} - \frac{a(k+l)}{A_{-}^2}
        \left(1- \tfrac{1}{\sigma}\right)\right)
      |\xi|^\frac{1}{\sigma} + \frac{a\xi}{2}.
    \end{split}
  \end{equation}
  Since $\psi(\xi) \sim \frac{a \xi}{2}$ as
  $\xi \rightarrow - \infty$, the coefficient of
  $|\xi|^\frac{1}{\sigma}$ vanishes and
  \begin{align}\label{eqn:AminFormWith_k_and_l}
    A_{-}^2 = a^{1- \frac{1}{\sigma}} (k+l) \left(1- \tfrac{1}{\sigma}\right) .
  \end{align}

  \begin{figure}[h]
    \includegraphics[height=8cm]{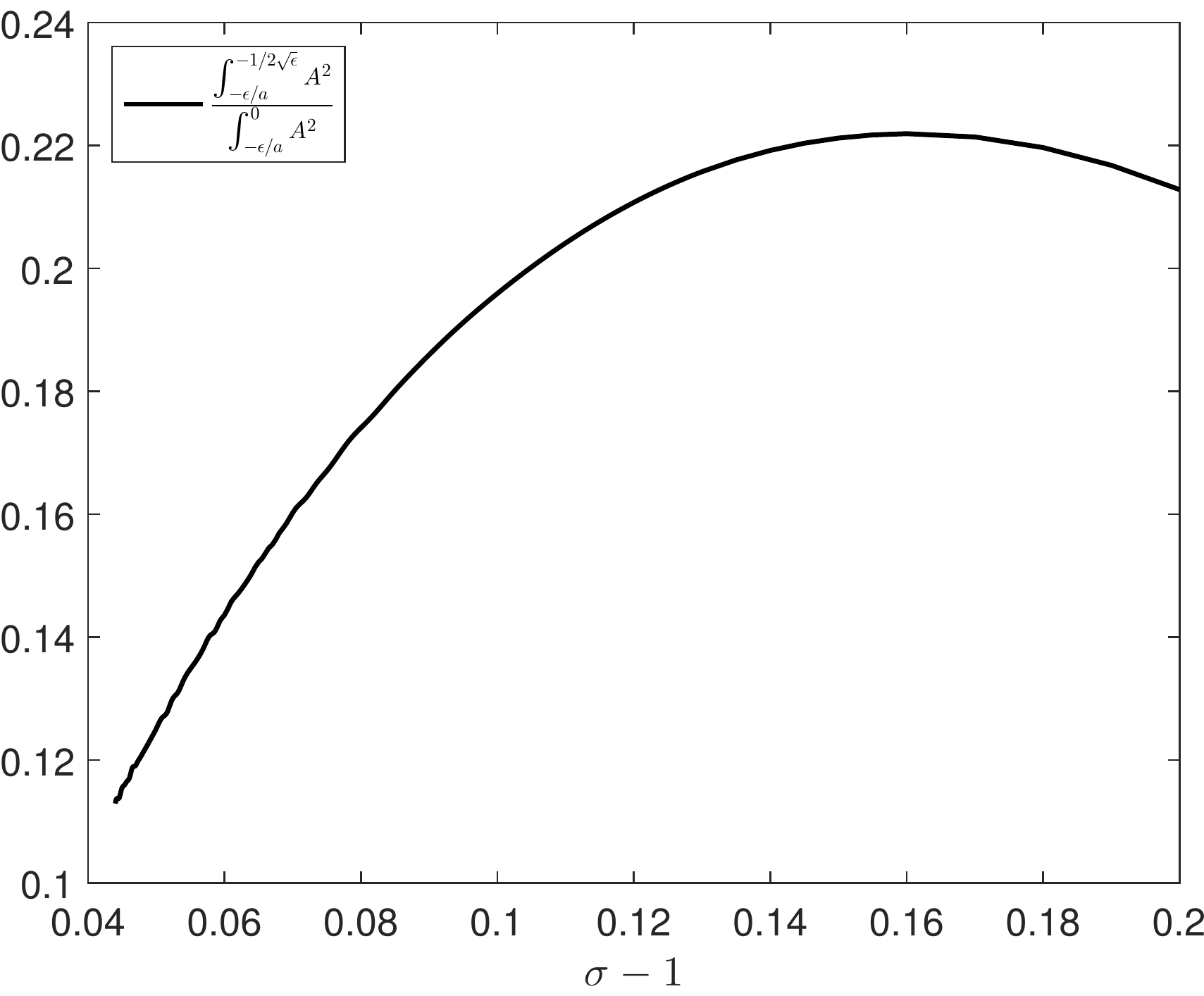}
    \caption{An estimate of the relative contribution of the region
      $\xi \in \bracket{-\eps/a, -1/{2\sqrt{\eps}}}$ (where the
      precise behaviour of the profile remains unknown) to the
      integral $l=\int_{-\eps/a}^{0} A^2$ appearing in
      \eqref{eqn:AminFormWith_k_and_l}}
    \label{fig:Amin_Integral_Relative_Soliton}
  \end{figure}

  In the limit $\sigma \rightarrow 1$, $k \rightarrow 2\pi$ (see
  Proposition \ref{prop:psi(0)}) while the main contribution to the
  integral $l = \int_{\eps/a}^0 A^2$ comes from the region
  $\paren{- \frac{1}{2\sqrt{\eps}}, 0}$ where the DNLS soliton
  \eqref{BrightSigma} approximates $A$ (see Figure
  \ref{fig:Amin_Integral_Relative_Soliton}). Therefore
  $l \approx \int_{-\infty}^{0} L_1^2 = 2\pi$ and the relation
  \eqref{eqn:AminAymptoticForm} follows.  In Figure
  \ref{fig:AminRelation}, we observe an excellent agreement of the
  numerical simulation with the formula
  \eqref{eqn:AminFormWith_k_and_l}. Similar to our result for
  $\psi(0)$, we check \eqref{eqn:AminFormWith_k_and_l} rather than
  \eqref{eqn:AminAymptoticForm} since for the values of $\sigma$ we
  computed, the integrals $k$ and $l$ have not reached their limiting
  values.
  \begin{figure}[h]
    \includegraphics[height=8cm]{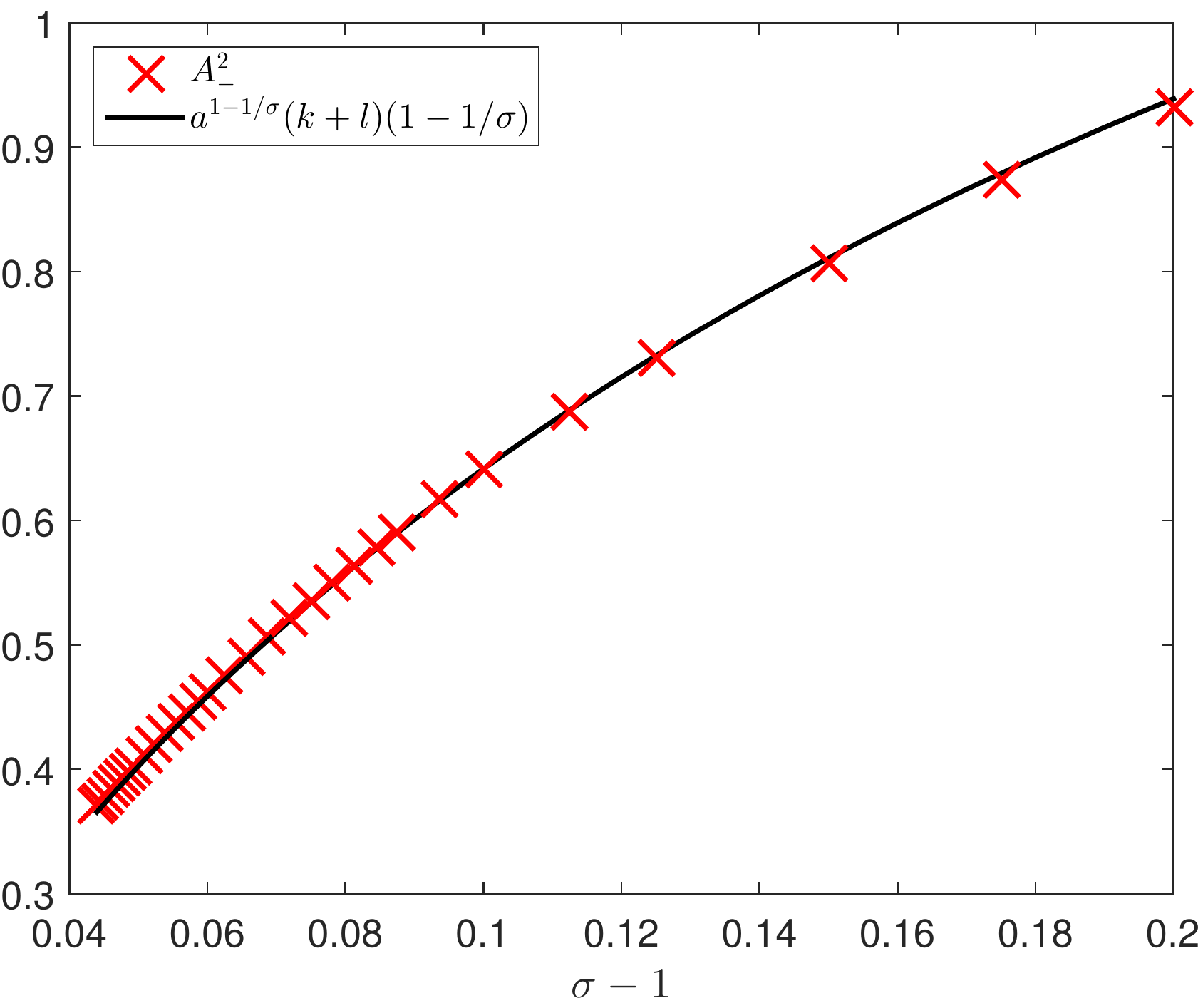}
    \caption{Numerical verification of
      \eqref{eqn:AminFormWith_k_and_l} for  $A_{-}$
      describing the asymptotic behaviour of $Q$ as
      $\xi \rightarrow - \infty$. }
    \label{fig:AminRelation}
  \end{figure}

\end{proof}

\subsubsection{Variation of turning point $\xi_-$ in terms of
  $\sigma$. } \label{sect:TPNumerics}

In the last section, we obtained the function $Q$ for negative values
of $\xi$ that satisfy conditions of validity for the WKB method,
namely $\xi < - \frac{\eps}{a}$ and
$|\xi + \frac{\eps}{a}| > a^{-1/3}$. We also know that for $\xi < 0$
with $|\xi| \ll \frac{\eps}{a}$, the amplitude is well approximated by
the DNLS soliton while the phase derivative $\psi$ remains small. In
order to match these behaviours we need to approximate $Q$ in the
intermediate region near $\xi \sim - \frac{\eps}{a}$.  Unlike nearby
the positive turning point $\xi_+ = \frac{4}{a}$, the problem here is
fully nonlinear. The equation satisfied by $P$ reduces to
\begin{align}\label{e:TPNum_LNLBalance}
  P_{\xi\xi} \approx \paren{a\xi + \eps} P + |P|^2 P
\end{align}
where both of the terms on the right hand side must be taken into
account. This equation can be transformed to one resembling a type II
Painlev\'e equation by setting $t = a^{-2/3} (a\xi + \eps)$ and
$u = a^{-1/3} 2^{-1/2}P$:
\[
u_{tt} = tu + 2 |u|^2 u.
\]
The nonlinearity however is of the form $|u|^2u$ rather than the
Painlev\'e $u^3$ and known results about approximate solutions to Painlev\'e do not apply.

Instead, we turn to our numerical data and examine the behaviour of
the turning point as a function of $\sigma$. We will show in Section
\ref{sect:Hamiltonian} that the parameter $\eps$ behaves as a power
law in $\paren{\sigma-1}$ in the limit $\sigma \to 1$ and therefore
make the Ansatz for $\xi_-$:
\begin{align}\label{e:TP_Model}
  \frac{a}{\eps} \approx C \paren{\sigma-1}^{\alpha}, \quad \sigma \to 1.
\end{align}
We use a standard least squares algorithm to compute $C$ and
$\alpha$ and find $C \approx 4$ while $\alpha \approx 1.2$.  Figure
\ref{fig:TP_GoF.} illustrates the goodness of the fit for
$\sigma \in \bracket{1.044, 1.1}$ with $C$ and $\alpha$ obtained from a Richardson extrapolation of the values of $a$ and $\eps$ from
simulations with $N = 2.56 \times 10^6$ and $N = 5.12 \times 10^6$
mesh points and $\sigma \in \bracket{1.044, 1.1}$. To check the
validity of this Ansatz, we change the range of $\sigma$ values
considered for the least square computation by restricting $\sigma$ to
$\bracket{1.044, \sigma_{max}}$ and varying $\sigma_{max}$. We do this
for data obtained from simulations performed at several different
resolutions and report the values obtained in Table
\ref{tab:TP_param}. In the worst case, we observe relative differences
in the values of $C$ and $\alpha$ at the order of $0.1 \%$.
  
\begin{figure}[h]
  \includegraphics[height=8cm]{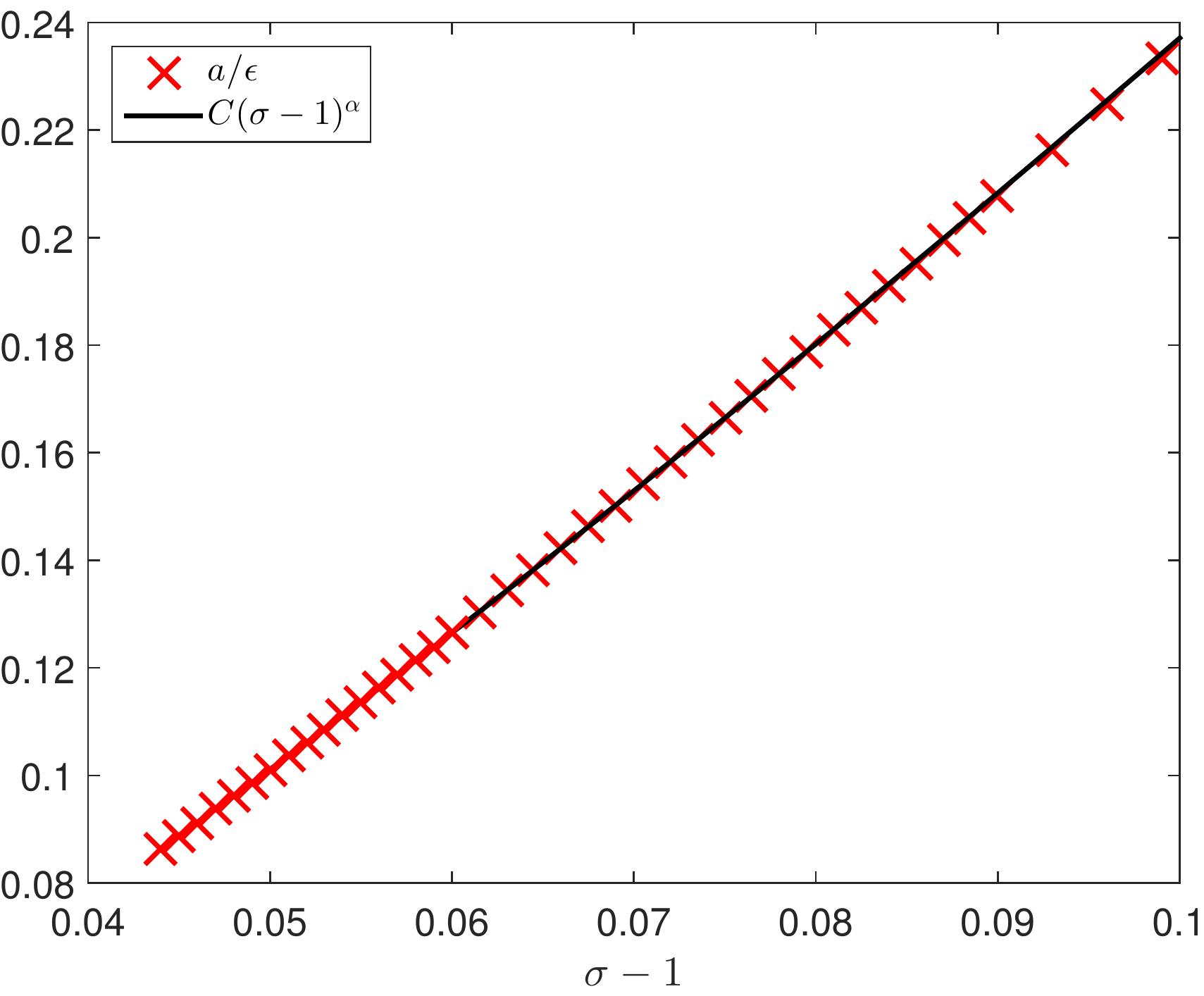}
  \caption{A numerical test of model \eqref{e:TP_Model} over a range
    of $\sigma$ values.  The values of $C$ and $\alpha$ were computed
    using a least square analysis Within this range of $\sigma$, we
    find $C \approx 4.03$ and $\alpha \approx 1.23$.  }
  \label{fig:TP_GoF.}
\end{figure}

\begin{table}[h]
  \caption{Computed values of parameters $\alpha$, $C$  in
    \eqref{e:TP_Model}. Left: using simulations with $N  = 5.12 \times
    10^6$ and $N = 2.56 \times 10^6$ mesh points. Right: using
    simulations with $N = 1.28 \times 10^6$ and $N = 2.56 \times 10^6$
    mesh points. 	 
  } 
  \label{tab:TP_param}

  \centering
  \pgfplotstableset{
    every head row/.style={before row=\toprule,after row=\midrule},
    every last row/.style={after row=\bottomrule}
  }
  \begin{filecontents*}{figs/TP_Param_Smax.txt}
smax alpha c
1.1 1.22545 3.97536
1.095 1.22595 3.98087
1.09 1.22653 3.98726
1.085 1.22729 3.9959
1.08 1.22808 4.00489
1.075 1.22894 4.0149
1.07 1.22983 4.02529
\end{filecontents*}
\pgfplotstabletypeset[
  1000 sep={\,},
  columns/smax/.style={
    fixed,fixed zerofill,precision=3,
    column type=r, column name = {$\sigma_{max}$}
  },
  columns/alpha/.style={
    fixed,fixed zerofill,precision=4,
    column type=r, column name = {$\alpha$}
  },
  columns/c/.style={
    fixed,fixed zerofill,precision=4,
    column type=r, column name = {$C$}
  }
  ]
  {figs/TP_Param_Smax.txt} \hspace{0.5cm} \begin{filecontents*}{figs/TP_Param_Smax_128_256.txt}
smax alpha c
1.1 1.22528 3.97371
1.095 1.22578 3.97914
1.09 1.22633 3.9853
1.085 1.22704 3.9933
1.08 1.22787 4.00275
1.075 1.2286 4.01124
1.07 1.22936 4.02014
\end{filecontents*}
\pgfplotstabletypeset[
  1000 sep={\,},
  columns/smax/.style={
    fixed,fixed zerofill,precision=3,
    column type=r, column name = {$\sigma_{max}$}
  },
  columns/alpha/.style={
    fixed,fixed zerofill,precision=4,
    column type=r, column name = {$\alpha$}
  },
  columns/c/.style={
    fixed,fixed zerofill,precision=4,
    column type=r, column name = {$C$}
  }
  ]
  {figs/TP_Param_Smax_128_256.txt}

  \vspace{0.3cm} 

\end{table}

\section{ Vanishing momentum condition}\label{sect:Hamiltonian}

In this Section, we use the zero momentum condition
\eqref{ZeroMomentumQ} to obtain an additional relation between $\eps$
and $\sigma$ in the limit $\sigma \to 1$.  Combined with
\eqref{e:TP_Model}, it gives the main conclusion of this study as
stated in equations \eqref{e:QSimLumpConclusion} and
\eqref{e:relations}.

\begin{proposition}\label{prop:Hamiltonian}
  As $\sigma \rightarrow 1$, the parameter $\epsilon$ satisfies, at
  leading order,
  \begin{align}\label{e:sqrtepsresult}
    \sqrt{\epsilon} \sim C \pi \left(\sigma - 1\right)
  \end{align}
  for some constant $2 < C < \frac{24}{7}$.
\end{proposition}
\begin{proof}
	
  In terms of phase and amplitude, $Q = A e^{i \theta}$, the property
  $I(Q) =0$ has the form
  \begin{align}\label{eqn:zeroenergy_phaseamp}
    I(Q) \equiv \int_{-\infty}^{\infty}\Id(Q)d\xi = \int_{-\infty}^{\infty}  \theta_\xi A^2 d\xi =0.
  \end{align}
  We separate the domain into three regions: {\it (i)}
  $- \infty < \xi \lessapprox \frac{ -\eps}{a}$; {\it (ii)} $\xi >0$;
  {\it (iii)} $\frac{ -\eps}{a} \lessapprox \xi < 0$ and denote $I_1$,
  $I_2$ and $I_3$ the corresponding contributions to $I$. In each
  region, we approximate the phase and amplitude of $Q$ using the
  analysis of the previous sections.
	 
  {\it Region 1:} When $-\infty < \xi \leq -\frac{ \epsilon}{a}$, we
  change variables to $y = -\frac{ 1}{2} (a\xi - b)$ and write
  \begin{equation}\label{eqn:H_1_definition}
    I_1 = \int_{-\infty}^{-\frac{\eps}{a}} \theta_\xi A^2 d\xi = \frac{2}{a} \int_1^\infty \theta_\xi A^2 dy 
  \end{equation}
	 
  For $y - 1 \gg a^\frac{2}{3}$, $A$ and $\theta_\xi$ are well
  approximated by \eqref{eqn:psiNegSqrt}. Since
  $\theta_\xi \approx \psi- \frac{1}{2}(a\xi -b) = \psi+ y$, we have
  \begin{align}\label{R1_approximations}
    A \approx \frac{a^\frac{1}{2\sigma} A_{-}}{\sqrt{2}(y^2 - 1)^\frac{1}{4}} \left(y+\sqrt{y^2-1}\right)^\frac{\sigma-1}{2\sigma} ; \quad 
    \theta_{\xi} \approx y - \sqrt{y^2 - 1} .
  \end{align}
  The contribution of the region
  $ 1 \leq y-1 \leq a^\frac{2}{3}$, equivalently
  $-\frac{\eps}{a} - a^{-\frac{1}{3}} \leq \xi \leq -\frac{\eps}{a}$
  (where the WKB analysis leading to \eqref{R1_approximations} is no
  longer valid) to $I$ is negligible compared to that of the region
  $y>1+a^\frac{2}{3}$. In Figure \ref{fig:R1_data_approx} we compare
  the values of $I_1$ obtained from the numerical integration of the
  solution to the boundary value problem (BVP) to those obtained by
  inserting \eqref{R1_approximations} into
  \eqref{eqn:H_1_definition}. We see a good agreement with a relative
  error of less than 1\%.  The leading order contribution to $I_1$ is
  therefore
  \begin{align}
    \frac{2}{a} \int_{1+a^\frac{2}{3}}^\infty \frac{\frac{a}{2} A_{-}^2 \left( y- \sqrt{y^2-1}\right)}{\sqrt{y^2-1}} dy \approx A_{-}^2. \nonumber
  \end{align}
  Using Propositon \ref{prop:Aneg} we have
  \begin{align}\label{R1_final}
    I_1 \approx 4\pi (\sigma- 1 ), \quad \sigma \to 1.
  \end{align}
  
  {\it Remark. } Our numerical integration of the BVP did not reach
  values of $\sigma$ sufficiently close to $1$ to allow a direct check
  of this relation.
  
  \begin{figure}[h]
    \centerline{\includegraphics[height=8cm]{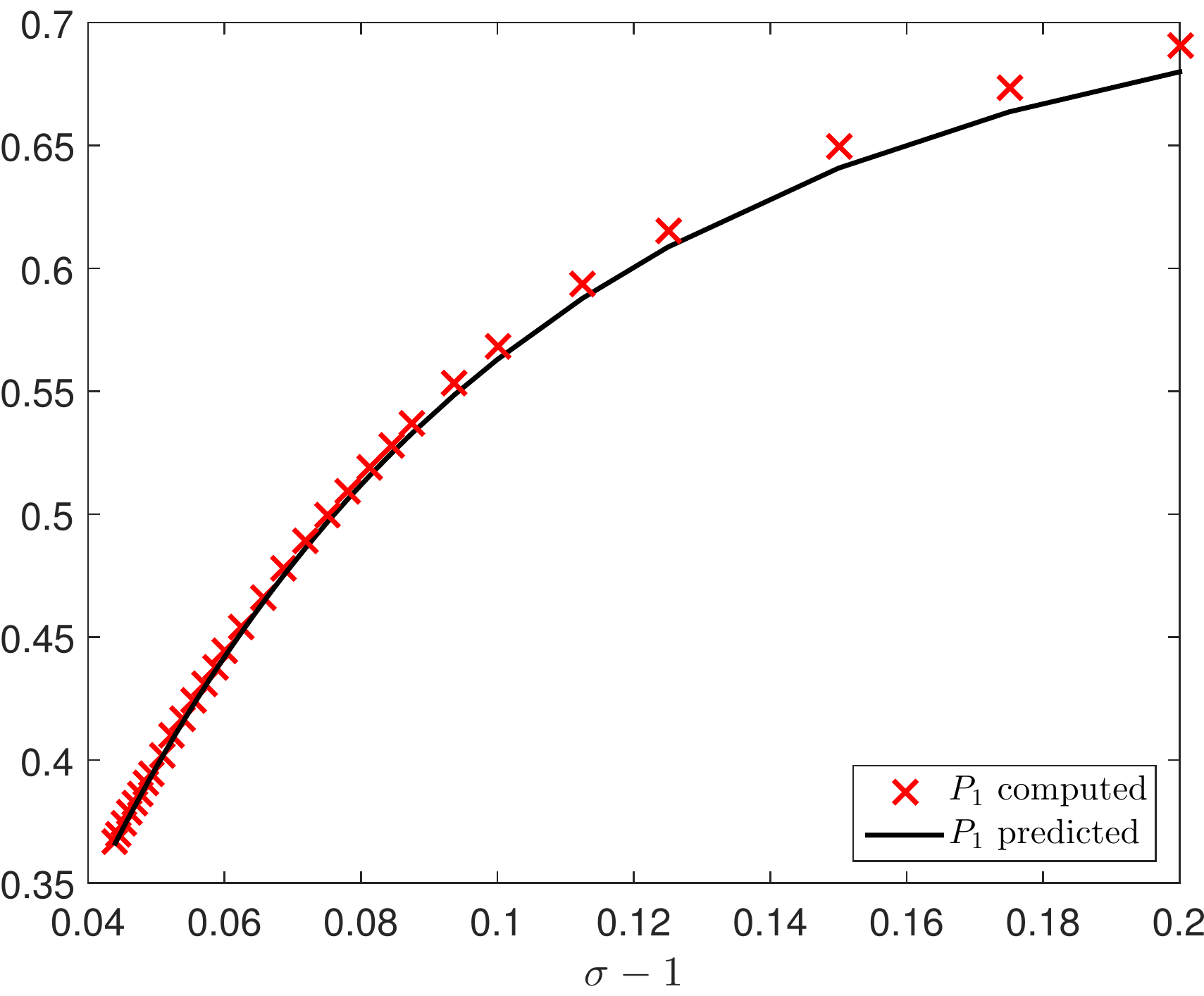}}
    \caption{Comparison of the values of $P_1$ obtained from the
      numerical simulation with those obtained from
      \eqref{R1_approximations}.}
    \label{fig:R1_data_approx}
  \end{figure}

  {\it Region 2:} When $\xi>0$, our simulations tell us that the
  amplitude $A$ is well approximated by the bright soliton
  \eqref{BrightSigma} as long as $\xi \ll \frac{ \eps}{a}$ with its
  region of validity extending at least to $\xi = a^\frac{ -1}{3}$ (
  see Figure \ref{fig:AmpBrightLump}).  For $\xi > \xi_+$, we can
  approximate $A$ using WKB, with an amplitude that is exponentially
  small as $a\to 0$ (see Proposition \ref{prop:APlusRelation}).
  Consequently, the contribution to $I_2$ of the entire region
  $\xi> a^{-1/3}$ is exponentially small; for
  $\xi \in (a^{-1/3}, \xi_+)$, it is small because
  \[
  A(\xi) \le B_\sigma(a^\frac{ -1}{3}) \approx \sqrt{\eps} \exp
  \left(-\sqrt{\eps} a^\frac{ -1}{3}\right),
  \]
  while for $\xi > \xi_+$, it will be small exponentially small
  because $A_+$ is small (Proposition \ref{prop:APlusRelation}).  In
  both regimes, we can make use of \eqref{e:TP_Model}, relating $a$ to
  $\eps$ and $\sigma-1$.  The consequence of this analysis is that the
  leading order contribution to $I_2$ is in $\xi < a^{-1/3}$, where
  the amplitude and phase derivative can be approximated by
  \begin{equation}\label{R3_approximations}
    A \approx B_\sigma = \paren{\frac{ (\sigma+1)(4-b^2)}{2 \left(\cosh \sigma \sqrt{4 - b^2}\xi - \frac{b}{2}\right)}}^\frac{1}{2\sigma}; \quad 
    \theta_{\xi} \approx \frac{b}{2} - \frac{1}{2\sigma+2} B_\sigma^{2\sigma}.
  \end{equation}
  We have omitted the $a\xi/2$ term from the phase derivative because
  it's contribution to the integral for $\xi <a^{-1/3}$ is
  $\bigo(a^{1/3})\ll \sqrt{\eps}$.  Under the assumption that
  $\sqrt{\eps} \propto \sigma-1$, which will be our conclusion,
  $\bigo(a^{1/3})\ll \sigma-1$, so the term will be small relative to
  the main contributions to the integral, which are
  $\bigo(\sqrt{\eps})$ and $\big(\sigma-1)$.  Thus, the contribution
  of this region to the momentum is approximated as
  \begin{equation}
    \begin{split}
      I_2\approx \int_0^{a^{-1/3}} \Id &\approx
      \int_0^{a^{-1/3}} \paren{\frac{b}{2} -
        \frac{1}{2\sigma+2} B_\sigma^{2\sigma}} B_\sigma^{2\sigma} \\
      &\approx \int_0^{\infty} \paren{\frac{b}{2} -
        \frac{1}{2\sigma+2} B_\sigma^{2\sigma}} B_\sigma^{2\sigma}
    \end{split}
  \end{equation}
  The final approximation is due to the integral over
  $(a^{-1/3}, \infty)$ of the approximate density being
  $\ll \sqrt{\eps}$.  Thus we include it for the convenience of
  analytical integration.

  Using these approximations, we expand $I_2$ as
  \begin{equation*}
    I_2 \approx  I_2(\sigma, \eps)\left.\right|_{\sigma=1} + (\sigma-1) \frac{\partial I_2}{\partial \sigma} (\sigma, \eps)  \left.\right|_{\sigma=1} 
  \end{equation*}
  By direct integration,
  \begin{equation}\label{e:H3_1term}
    I_2(\sigma, \eps)\left. \right|_{\sigma = 1} = \int\limits_{0}^{\infty} \left( \frac{b}{2} - \frac{B_1^2}{4} \right) B_1^2 d\xi = -\sqrt{4-b^2} \approx -2 \sqrt{\eps}.
  \end{equation}
  The second term in the expansion is given by 
    \begin{align}
      \left. \frac{\partial I_2}{\partial \sigma} \right|_{\sigma=1} \approx \int_0^\infty \left. \frac{\partial}{\partial \sigma} \bracket{\paren{\frac{b}{2} - \frac{B_\sigma^{2\sigma}}{2\sigma+2}} B_\sigma^2} \right|_{\sigma=1} d\xi.
    \end{align}
    To approximate this integral, we first claim that the main
    contribution comes from
    $\xi < \frac{1}{\sqrt{4-b^2}} \approx \frac{1}{2\sqrt{\eps}}$. To
    see this, we note that the bright soliton $B_\sigma$ is a function
    of $u = \sigma \sqrt{4-b^2} \xi \approx 2\sqrt{\eps}\xi$ and take
    great care when differentiating with respect to $\sigma$ under the
    integral sign. To wit, we split the integral into 2 parts at
    $\xi_0 = \frac{1}{\sigma \sqrt{4-b^2}}$ and write
    \begin{equation}
      \begin{split}
        \left. \frac{\partial I_2}{\partial \sigma} \right|_{\sigma=1}
        &= \underbrace{\int_0^{\xi_0} \left. \frac{\partial}{\partial
              \sigma} \bracket{\paren{\frac{b}{2} -
                \frac{B_\sigma^{2\sigma}}{2\sigma+2}} B_\sigma^2}
          \right|_{\sigma=1}d\xi }_{\equiv I_{2,1} }\\
        &\quad + \underbrace{\xi_0^{-1} \int_1^\infty
          \left. \frac{\partial}{\partial \sigma}
            \bracket{\paren{\frac{b}{2} -
                \frac{B_\sigma^{2\sigma}}{2\sigma+2}} B_\sigma^2}
          \right|_{\sigma=1} du }_{\equiv I_{2,2}}.
      \end{split}
    \end{equation}
    We now observe that the second integral, $I_{2,2}$, tends to zero
    as $\sigma \to 1$ while the first tends to a finite value. Indeed,
    via direct computation, we obtain
    \begin{equation}
      \label{e:InterOrderTerm}
      \begin{split}
        I_{2,2} &\approx \frac{1}{2\sqrt{\eps}} \int_1^\infty \paren{\frac{b}{2} - \frac{B_1^2}{4}} \paren{\frac{1}{2} B_1^2 - 2 B_1^2 \log B_1} du \\
        &\approx \frac{1}{\sqrt{\eps}} \paren{c_1 \eps \log \eps + c_2 \eps + O(\eps^2)} \\
        &\approx c_1 \sqrt{\eps} \log \eps + c_2 \sqrt{\eps}
      \end{split}
    \end{equation}
    where $B_1$ is the bright soliton with $\sigma = 1$ and
    $c_1 \approx -2.33$ and $c_2 \approx -0.66$ are constant
    values. Under the Ansatz that $\eps$ behaves as a power law in
    $\paren{\sigma-1}$, $I_{2,2}$ tends slowly to zero. For $I_{2,1}$,
    the bright soliton nearly coincides with the lump and we have by
    direct computation
    \begin{equation}
      \label{e:H3_2term}
      \begin{split}
        I_{2,1} &\approx \int_0^{1/{2\sqrt{\eps}}} \left\{\frac{B_1^6 \xi \sinh\sqrt{4-b^2}\xi}{4\sqrt{4-b^2}} \right. \\
        &\quad\quad \left. + \paren{\frac{b}{2} -
            \frac{B_1^2}{4}} \paren{-2 B_1^2 \ln B_1 + \frac{B_1^2}{2}
            -\frac{B_1^4 \xi \sinh \sqrt{4-b^2} \xi}{\sqrt{4-b^2}}}
        \right\}d\xi.  \\
        &\approx \int_0^\infty \set{ \frac{\xi^2 L_1^6}{4} + \left(1- \frac{L_1^2}{4}\right) \left(-2 L_1^2 \ln L_1 + \frac{L_1^2}{2} - \xi^2 L_1^4\right)} d\xi \\
        &= 2\pi,
      \end{split}
    \end{equation}
    where we take the limit $\eps \to 0$ in the penultimate step.
   Using \eqref{e:H3_1term} and \eqref{e:H3_2term}, we conclude that
  \begin{equation}\label{R3_final}
    I_2 \approx -2 \sqrt{\eps} + 2\pi (\sigma -1), \quad \sigma \to 1.
  \end{equation}

  {\it Region 3:} Near the origin, for $\xi < 0$, the amplitude is
  well approximated by the bright soliton.  However, as we approach
  $\xi_-$, the linear term $(a\xi + \eps) P$ in \eqref{PEquation}
  becomes less relevant and the lump soliton becomes a better
  approximation . We thus subdivide the integral $I_3$ into the
  regions $-\frac{1}{2\sqrt{\eps}} < \xi < 0$ and
  $-\frac{\eps}{a} < \xi < -\frac{1}{2\sqrt{\eps}}$.

\begin{figure}[h]
  \centerline{\includegraphics[height=8cm]{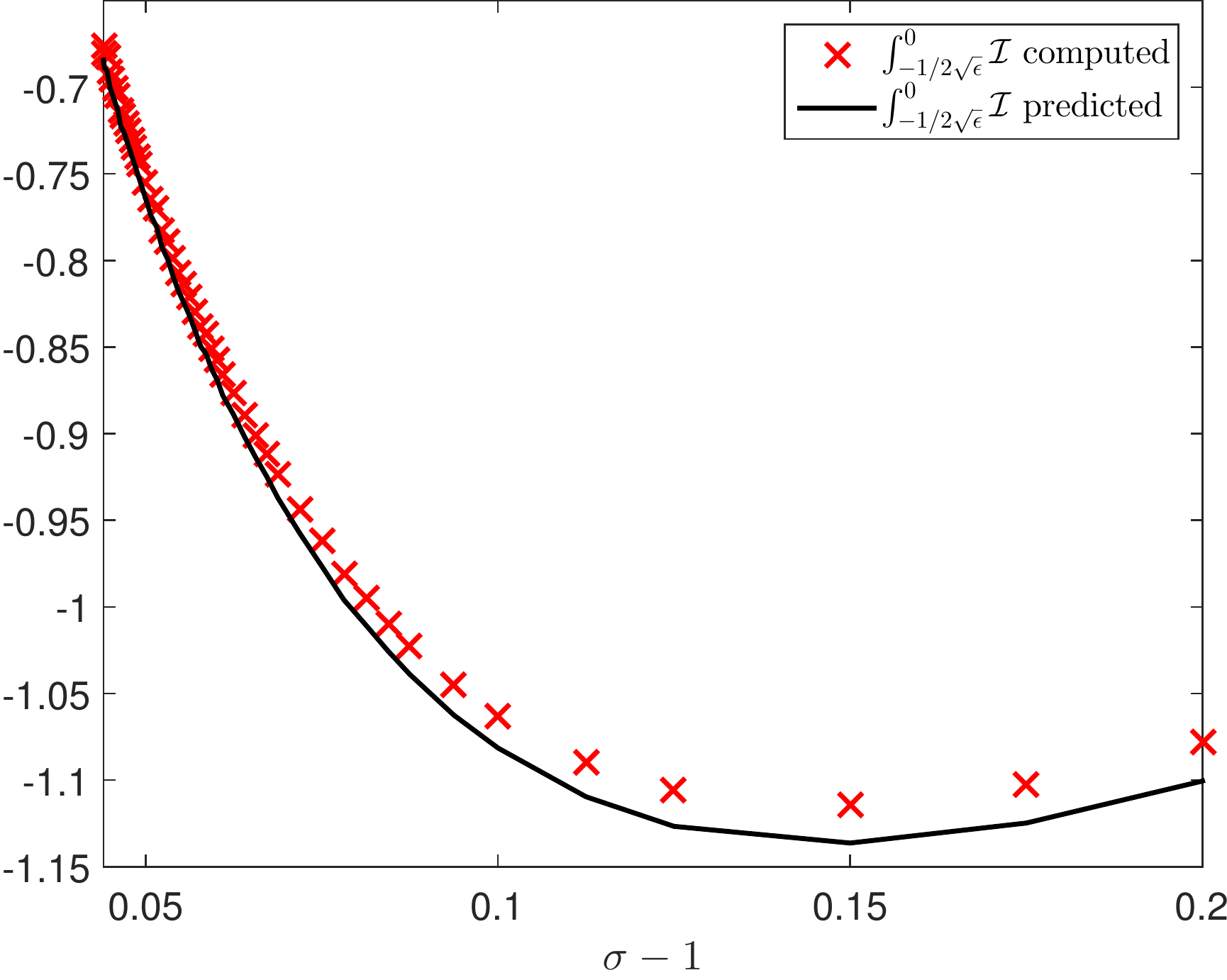}}
  \caption{A comparison of the contribution to the momentum $P$ of the
    region $[-\frac{1}{2\sqrt{\eps}},0]$ from the numerical solution
    to that predicted by the bright soliton approximation.}
  \label{fig:Bright_neg_validity}
\end{figure}

Figure \ref{fig:Bright_neg_validity} compares the contribution to the
momentum, over the interval $-\frac{1}{2\sqrt{\eps}} < \eps <0$,
between the numerical solution and approximation
\eqref{R3_approximations}, using the bright soliton. We find that they
are in good agreement over a range of values of $\sigma$ with a
relative error of less than 2\%. Therefore we approximate the
contribution of this interval by calculating
$\int_{-\frac{1}{2\sqrt{\eps}}}^{0} \Id(B_\sigma)$ to leading
order. When $-\frac{1}{2\sqrt{\eps}} < \xi <0$, we expand the bright
soliton near $\sigma =1,\eps =0$ as

\begin{align}\label{B_expansion_R2}
  B_\sigma (\xi) = L_1(\xi) + (\sigma-1)f_1(\xi) + \eps f_2(\xi)
\end{align}
where $L_1$ is the lump soliton ($\sigma=1$) and $f_1$, $f_2$ are
given by
\begin{equation}
  \label{A1_A2}
  \begin{split}
    f_1(\xi) &= -\frac{1}{\sqrt{2}}\left(\frac{1}{4 \xi^2+1}\right)^{3/2} \bracket{12 \xi^2+\left(8 \xi^2+2\right) \log \left(\frac{8}{4 \xi^2+1}\right)-1}, \\
    f_2(\xi) &= -\frac{16 \xi^4+3}{6 \sqrt{2} \left(4
        \xi^2+1\right)^{3/2}}.
  \end{split}
\end{equation}
The contribution of this interval to $I_3$ becomes
\[
\int\limits_{-\frac{1}{2\sqrt{\eps}}}^{0} \Id(B_\sigma) \approx
\int\limits_{-\frac{1}{2\sqrt{\eps}}}^{0}\Id_0 + (\sigma-1)
\int\limits_{-\frac{1}{2\sqrt{\eps}}}^{0} \Id_1 +
\eps\int\limits_{-\frac{1}{2\sqrt{\eps}}}^{0} \Id_2
\]
where the integrands are computed at leading order using
\eqref{B_expansion_R2}:
\begin{align*}
  \Id_0 &= \left(1-\frac{L_1^2}{4}\right) L_1^2 \\
  \Id_1 &=2 L_1 f_1 \left(1-\frac{L_1^2}{4}\right) - \frac{1}{2} \left(L_1 f_1 + L_1^2 \ln L_1 - \frac{L_1^2}{4}\right) L_1^2 \\
  \Id_2 &= 2 f_2 L_1 \left(1- \frac{L_1^2}{4}\right) - \frac{1}{2} \left(f_2 L_1 +1\right) L_1^2. 
\end{align*}
Using Mathematica, we find:
\begin{align}
  \int_{-\frac{1}{2\sqrt{\eps}}}^{0}\Id_0  \approx -4\sqrt{\eps}, \quad \int_{-\frac{1}{2\sqrt{\eps}}}^{0}\Id_1 \approx 2\pi, \quad \int_{-\frac{1}{2\sqrt{\eps}}}^{0}\Id_2 \approx -\frac{1}{3\sqrt{\eps}}.
\end{align}  
To summarize, we have
\[
\int_{-\frac{1}{2\sqrt{\eps}}}^0 \Id(B_\sigma) \approx -
\frac{13}{3}\sqrt{\eps} + 2\pi(\sigma-1).
\]

\begin{figure}[h]
  \centerline{\includegraphics[height=8cm]{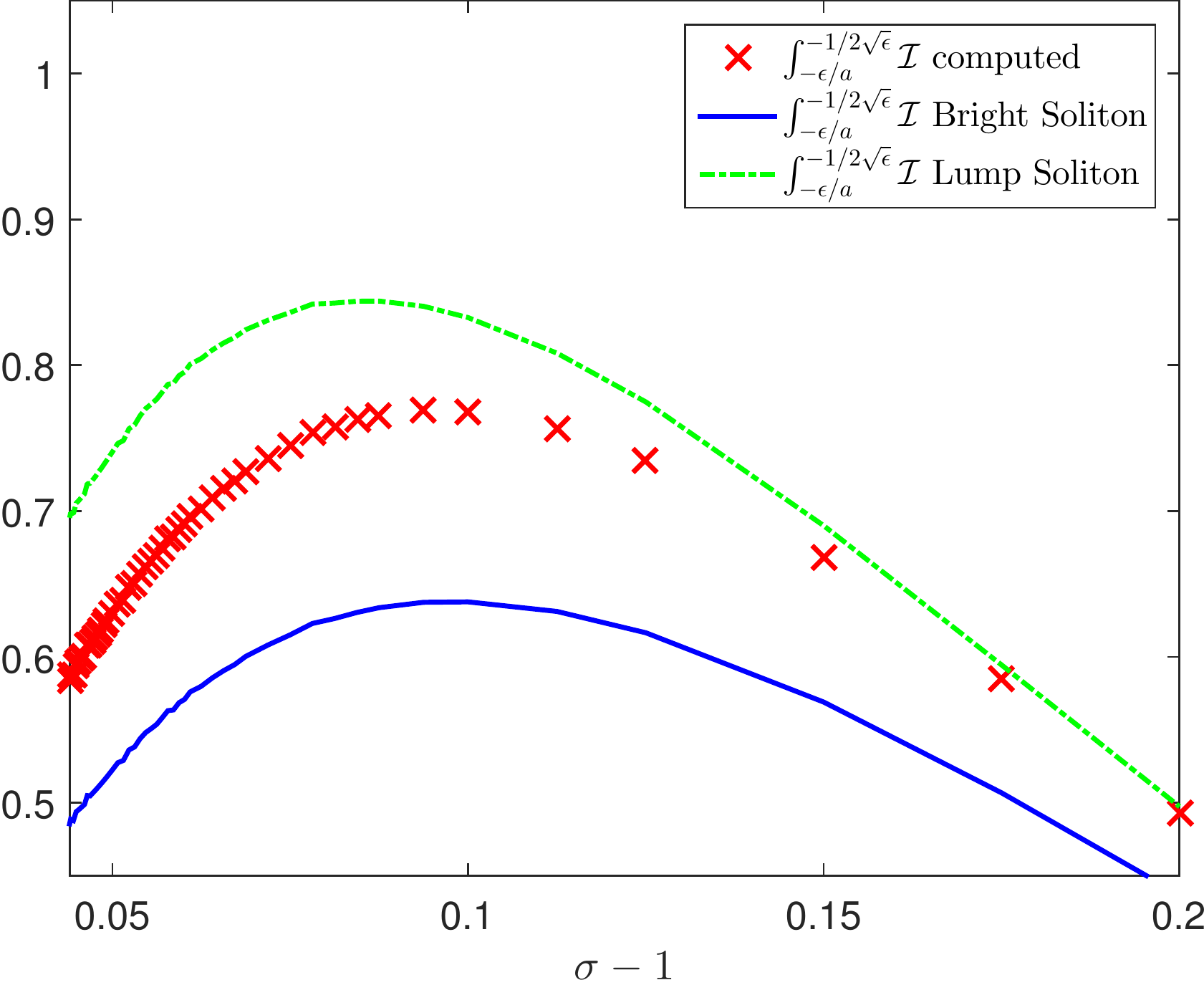}}
  \caption{Comparison of the contribution to the momentum of the
    region $[-\frac{\eps}{a}, -\frac{1}{2\sqrt{\eps}}]$ obtained in
    the simulation with the same quantity using the bright soliton and
    using the lump soliton approximations.}
  \label{fig:R2_Lump_Bright}
\end{figure}

For $-\frac{\eps}{a} < \xi < -\frac{1}{2\sqrt{\eps}}$, the bright
soliton is not a valid approximation of the amplitude. Although we do
not have a precise analytic expression of $Q$ throughout this region,
we observe, as illustrated in Figure \ref{fig:R2_Lump_Bright}, that
\[
\int_{-\frac{\eps}{a}}^{-\frac{1}{2\sqrt{\eps}}} \Id(L_\sigma) >
\int_{-\frac{\eps}{a}}^{-\frac{1}{2\sqrt{\eps}}}\Id >
\int_{-\frac{\eps}{a}}^{-\frac{1}{2\sqrt{\eps}}} \Id(B_\sigma).
\]
When $\sigma \to 1$, $L_\sigma \approx L_1 + (\sigma -1)f_1$ with
$f_1$ as in \eqref{A1_A2} and
\[
\int_{-\frac{\eps}{a}}^{-\frac{1}{2\sqrt{\eps}}} \Id(L_\sigma) \approx
4\sqrt{\eps}.
\]

On the other hand,
$\int_{-\frac{\eps}{a}}^{-\frac{1}{2\sqrt{\eps}}} \Id(B_\sigma)$ can
be estimated by the same methods as in Region 2.  Indeed, recognizing
that the density $\Id(B_\sigma)$ is even in $\xi$,
\[
\int_{-\frac{\eps}{a}}^0 \Id(B_\sigma) = \int_0^{\frac{\eps}{a}}
\Id(B_\sigma) \approx \int_0^{\infty} \Id(B_\sigma)\approx 2\pi
(\sigma-1) - 2\sqrt{\eps}.
\]
Combining the bounds and estimates for the two pieces of Region 3,
\[
\int_{-\frac{\eps}{a}}^{-\frac{1}{2\sqrt{\eps}}} \Id(B_\sigma)\approx
\frac{7\sqrt{\eps}}{3},
\]
which gives us
\begin{align}\label{R2_final}
  I_2 \sim 2\pi(\sigma-1) - \mu  \sqrt{\eps}
\end{align}
where $\mu$ is some constant satisfying $\frac{1}{3} < \mu < 2$.

Combining the 3 regions, namely \eqref{R1_final}, and \eqref{R3_final}
and \eqref{R2_final}, we find that
\[
I(Q) \sim 8\pi(\sigma-1) - (\mu +2) \sqrt{\eps}.
\]
We now impose the vanishing momentum condition to get the main result
of this section:
\begin{align}\label{e:MMMainresult}
  \sqrt{\eps} \sim \frac{8 \pi}{\mu + 2} (\sigma-1), \quad 1/3 < \mu < 2.
\end{align}
We are unable to numerically check the leading order behaviour of
$\eps$ in Eq. \eqref{e:MMMainresult} directly because we have not
reached values of $\sigma$ sufficiently close to $1$ to ignore higher
order corrective terms on the right hand side.  Indeed, as
  exhibited in \eqref{e:InterOrderTerm}, the corrections may be of
  intermediate order, and they significantly affect the numerics. For
  this reason, we include a term of the form
  $f(\sigma) \sim \paren{\sigma-1}^2 \log \paren{ \sigma-1 }$ in the
  right hand side of \eqref{e:MMMainresult} and make the Ansatz 
\begin{align}\label{e:NLLSQ_sqrteps}
  \sqrt{\eps} = \paren{C_0 + C_1 \paren{\sigma-1} \log \paren{\sigma-1}} \paren{\sigma-1}^\alpha.
\end{align}
We then use a nonlinear least squares algorithm to calculate
$C_0$, $C_1$ and $\alpha$.  We find that $C_0 \approx 8$,
$C_1 \approx 15$ and $\alpha \approx 1$ (the latter being the result
derived in \eqref{e:MMMainresult} analytically). The value
$C_0 \approx 8$ corresponds to $\mu \approx 1.1$ in
\eqref{e:MMMainresult}, well within the predicted range. Figure
\ref{fig:SqrtEps_GoF} illustrates the goodness of the fit of
\eqref{e:MMMainresult} for $\sigma \in \bracket{1.035, 1.1}$ and
values of $C_0$, $C_1$, and $\alpha$ obtained by a least square
analysis of Richardson extrapolation of $\eps$ values from
computations using $N = 2.56 \times 10^6$ and $N = 5.12 \times 10^6$
mesh points. To check the validity of the obtained values we proceed
as for the model \eqref{e:TP_Model}: we restrict the values of
$\sigma$ considered in the least square analysis to
$\sigma \in \bracket{1.044, \sigma_{max}}$ and vary $\sigma_{max}$.
We also use results from computations performed at different
resolutions and report the values obtained in Table
\ref{tab:sqrteps_params}. In the worst case, we observe a relative
difference between the obtained values on the order of $4\%$.

\begin{figure}[h]
  \includegraphics[height=8cm]{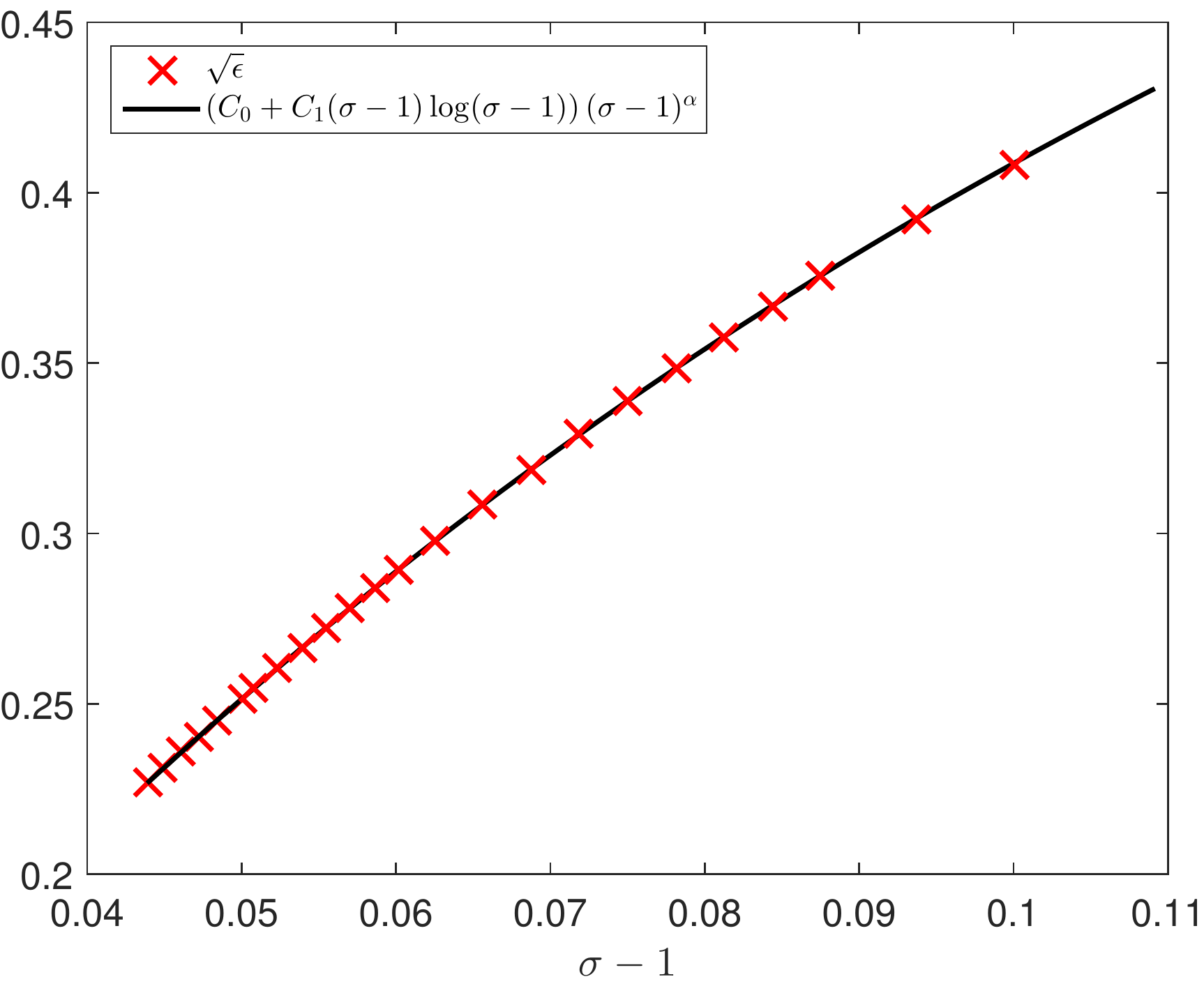}
  \caption{Numerical verification of model \eqref{e:NLLSQ_sqrteps}
    using a least square computation to find the parameters
    $C_0$,
    $C_1$,
    and $\alpha$.
    We find that $C_0
    \approx 7.915$, $C_1 \approx 14.859$ and $\alpha \approx 1.041$. }
  \label{fig:SqrtEps_GoF}
\end{figure}

\begin{table}[h]
  \caption{A table of computed values for the parameters $\alpha$, $C_0$ and $C_1$ in \eqref{e:NLLSQ_sqrteps}. Left: using simulations with $N = 5.12 \times 10^6$ and $N = 2.56 \times 10^6$ mesh points. Right: using simulations with $N = 1.28 \times 10^6$ and $N = 2.56 \times 10^6$ mesh points. 
  }
  \label{tab:sqrteps_params}

  \centering
  \pgfplotstableset{
    every head row/.style={before row=\toprule,after row=\midrule},
    every last row/.style={after row=\bottomrule}
  }
  \begin{filecontents*}{figs/RootEpsParam.txt}
smax alpha c1 c2 
1.1 1.04517 8.03658 15.2261
1.095 1.04464 8.02006 15.1771
1.09 1.04411 8.00368 15.1283
1.085 1.04349 7.98433 15.0701
1.08 1.04282 7.96311 15.0056
1.075 1.04214 7.94153 14.9395
1.07 1.04132 7.9154 14.8587
\end{filecontents*}
\pgfplotstabletypeset[
  1000 sep={\,},
  columns/smax/.style={
    fixed,fixed zerofill,precision=3,
    column type=r, column name = {$\sigma_{max}$}
  },
  columns/alpha/.style={
    fixed,fixed zerofill,precision=3,
    column type=r, column name = {$\alpha$}
  },
  columns/c1/.style={
    fixed,fixed zerofill,precision=3,
    column type=r, column name = {$C_0$}
  },columns/c2/.style={
    fixed,fixed zerofill,precision=3,
    column type=r, column name = {$C_1$}
  }
  ]
  {figs/RootEpsParam.txt} \hspace{0.5cm} \begin{filecontents*}{figs/RootEpsParam128_256.txt}
smax alpha c1 c2 
1.1 1.04439 8.01346 15.1608
1.095 1.04375 7.99396 15.103
1.09 1.04306 7.97235 15.0387
1.085 1.04219 7.94531 14.9574
1.08 1.04145 7.92213 14.887
1.075 1.04045 7.89057 14.7904
1.07 1.03929 7.85397 14.6774
\end{filecontents*}
\pgfplotstabletypeset[
  1000 sep={\,},
  columns/smax/.style={
    fixed,fixed zerofill,precision=3,
    column type=r, column name = {$\sigma_{max}$}
  },
  columns/alpha/.style={
    fixed,fixed zerofill,precision=3,
    column type=r, column name = {$\alpha$}
  },
  columns/c1/.style={
    fixed,fixed zerofill,precision=3,
    column type=r, column name = {$C_0$}
  },columns/c2/.style={
    fixed,fixed zerofill,precision=3,
    column type=r, column name = {$C_1$}
  }
  ]
  {figs/RootEpsParam128_256.txt}

  \vspace{0.3cm} 

\end{table}

\end{proof}

{\it Remark:} It is possible to prove Proposition
\ref{prop:Hamiltonian} by considering the condition $H(Q) = 0$ (see
\eqref{ZeroEnergyQ} in place of $P(Q) = 0$. Using the same analysis and separation of the
domain, we obtain $H_1 \sim 2\pi (\sigma-1)$,
$H_2 \sim 3\pi(\sigma-1) - 2 \sqrt{\eps}$ and
$H_3 \sim 3\pi (\sigma-1) - \mu \sqrt{\eps}$. We omit the details of
this calculation as the Hamiltonian density is a more complicated
object while the final result \eqref{e:sqrtepsresult} is unchanged.

\section{Discussion and further remarks}

Combining Proposition \ref{prop:Hamiltonian} and the numerical fit
\eqref{e:TP_Model} gives us the central result of this study: that in the limit
$\sigma \to 1$, $a(\sigma)$ behaves as a power law with respect to the
distance to criticality $(\sigma-1)$ (that is
$a \sim \paren{\sigma-1}^\alpha$ with $\alpha \approx 3.2$), while the
amplitude of the blow-up profile tends to the lump soliton of DNLS. In
the course of the analysis, we have made several assumptions on the
relative behaviours of $a$ and $\epsilon$ that we now check a
posteriori. We have assumed that $a^{2/3} \ll \epsilon$,
$\frac{a}{\epsilon} \ll \sigma - 1$ and $\epsilon \ll (\sigma -1)$,
all consistent with our final result.

In some cases, we did not check the asymptotic relations directly
against our numerical simulation because although we reach values of
$\sigma$ as low as $\sigma = 1.044$, we cannot ignore some of the higher
order corrections.  For example, in Propositions
\ref{prop:APlusRelation}, \ref{prop:psi(0)} and \ref{prop:Aneg}, we
derived the form of the parameters $A_+$, $\psi(0)$ and $A_-$ (through
equations \eqref{eqn:AplusFinalRelation}, \eqref{eqn:psi0Form_with_k}
and \eqref{eqn:AminFormWith_k_and_l}) that have more than leading
order precision and we find excellent agreement with the numerical
simulations. In Section \ref{sect:TPNumerics} however, we were
restricted to a heuristic discussion in a neighborhood of the turning
point $\xi_- = -\frac{\eps}{a}$. The behaviour of the profile here is
a result of a delicate balance of linear and non-linear terms in
\eqref{e:TPNum_LNLBalance} and its precise analytic description
remains an open problem. Finally, in Proposition
\ref{prop:Hamiltonian}, we estimated the integral constraint
$P = \int \theta_\xi A^2 =0$ using the approximations of the preceding
sections. We use the DNLS solitons \eqref{LumpSigma} and
\eqref{BrightSigma} to bound explicitly above and below the integral
over the neighborhood of $\xi_-$ where nonlinear effects are
important. We obtain relation \eqref{e:sqrtepsresult}; however, the
constant of proportionality is not known precisely.

\appendix

\section{Details of the Asymptotic Expansion}
\label{App:Asymptotics}
This Appendix contains the proof of Proposition \ref{prop:Asymptotics}
which is a slight extension of Proposition 4.1 in \cite{Liu:2013ej}.
We decompose the blowup profile as $Q = XZ$, where $X$ is a phase term
chosen to remove linear terms in $Z_\xi$.  Let
\begin{equation}
  X(\xi)  = \exp\set{-i \left(\frac{a\xi^2}{4} - \frac{b\xi}{2} + \frac{1}{2}
      \int_{0}^{\xi} |Z(\xi')|^{2\sigma} d\xi'\right)},
\end{equation}
$Z$ satisfies
\begin{equation}
  \label{eqn:ZEqn}
  Z_{\xi\xi} + \left(\frac{1}{4} \left(a\xi - b \right)^2 +
    \frac{1}{2}(a\xi -b)|Z|^{2\sigma} + \frac{|Z|^{4\sigma}}{4}-1 - i
    \frac{a(\sigma-1)}{2\sigma} -
    \frac{i}{2}\left(|Z|^{2\sigma}\right)_\xi\right)Z = 0 .
\end{equation}
Decomposing $Z$ into phase and amplitude, $Z=A e^{i\phi}$, gives
\begin{gather}
  \label{eqn:ZAmpEqn}
  \frac{A_{\xi\xi}}{A} - \phi_\xi^2 - 1 + \frac{1}{4}(a\xi - b)^2 +
  \frac{1}{2}(a\xi - b)A^{2\sigma} + \frac{A^{4\sigma}}{4} =0\\
  \label{eqn:ZPhaseEqn}
  \phi_{\xi\xi} + 2 \frac{A_{\xi}}{A} \phi_\xi -
  \frac{a(\sigma-1)}{2\sigma} A - \frac{1}{2}
  \left(A^{2\sigma}\right)_\xi = 0 .
\end{gather}
Let $\theta \equiv \phi_\xi$.  Following \cite{Liu:2013ej}, we now
assume that, as $\xi\to \pm \infty$,
\begin{align}
  \label{e:theta_approx}
  \theta(\xi) &= \frac{a\xi-b}{2} - \frac{1}{a\xi} -
                \frac{b^2}{a^2\xi^2}
                + \frac{1}{2}A^{2\sigma} + \gamma(\xi), \quad \gamma(\xi) = \bigo(\xi^{-3}),\\
  \label{e:A_approx}
  A(\xi) & = A_\pm \abs{\xi}^{-\frac{1}{2\sigma}} \paren{1 +
           \frac{b}{2a \sigma \xi} + \nu(\xi)}, \quad \nu(\xi) =
           \bigo(\xi^{-2}).
\end{align}
$\gamma(\xi)$ and $\nu(\xi)$ are corrections to the terms explicitly
written.  While we have made an assumption as to their order as
$\xi \to \pm \infty$, they remain undetermined at this point.  We will
also assume that they are smooth, and their derivatives obey
\[
\gamma^{(n)}(\xi) =\bigo(\xi^{-3-n}), \quad \nu^{(n)}(\xi) =
\bigo(\xi^{-2-n}).
\]

We substitute \eqref{e:theta_approx} and \eqref{e:A_approx} into
\eqref{eqn:ZAmpEqn} and \eqref{eqn:ZPhaseEqn}. We must show that the
corrections, with the assumed orders, are consistent; there must be
other terms in the equations which can balance them.  Then, in
principle, we could successively solve for the next correction.  One
subtlety is that in \eqref{e:theta_approx} and in the terms
$A^{2\sigma}$ and $A^{4\sigma}$ in \eqref{eqn:ZAmpEqn}, we will not
immediately make use of \eqref{e:A_approx}.  The reason for this is
that a number of the terms cancel exactly, leading to simpler
equations.  For the amplitude equation, \eqref{eqn:ZAmpEqn}, we obtain
\begin{equation} \label{A5} \underbrace{\frac{A''}{A} -
    \frac{b^2}{a^2\xi^2}- \frac{1}{a^2\xi^2} + \frac{1}{a\xi}
    A^{2\sigma}}_{\bigo(\xi^{-2})} - a\xi \gamma(\xi)=
  \bigo(\xi^{-3}).
\end{equation}
One can check that the indicated terms are of order $\xi^{-2}$.  Since
we have assumed that $\gamma(\xi) = \bigo(\xi^{-3})$,
$a\xi \gamma(\xi)$ will be $\bigo(\xi^{-2})$, and thus it is
consistent.  We could obtain the leading order $\xi^{-3}$ term in
$\gamma$, but we do not pursue this.  The right-hand-side of \eqref{A5}
contains a number of terms that can be checked to be of order at least
$\xi^{-3}$.

Turning to \eqref{eqn:ZPhaseEqn}, we will explicitly retain all terms
of order $\xi^{-2}$, and verify that $\nu(\xi)$ appears at the correct
order.  We first expand $A_\xi/A$ using \eqref{e:A_approx}, to obtain
\begin{equation*}
  \frac{A_\xi}{A} = -\frac{1}{2\sigma \xi} - \frac{b}{2a\sigma \xi^2} +
  \frac{b^2}{4 a^2 \sigma^2 \xi^3} + \nu_\xi(\xi) + \bigo(\xi^{-4}).
\end{equation*}
Under our assumption on $\nu$, $\nu_\xi$ is of order $\xi^{-3}$. Then,
substituting in the above expression into \eqref{eqn:ZPhaseEqn}
\begin{equation}\label{A6}
  \underbrace{\frac{1}{a\xi^2} + \frac{b^2}{4 a \sigma^2 \xi^2} + \frac{1}{a\sigma
      \xi^2} + \frac{b^2}{2a \sigma \xi^2} - \frac{1}{2\sigma\xi
    }A^{2\sigma}}_{\bigo(\xi^{-2})}  + a\xi \nu_\xi(\xi) = \bigo(\xi^{-3}).
\end{equation}
The indicated terms on the left-hand-side of \eqref{A6} are all of
order $\xi^{-2}$.  Under the assumption on $\nu$ and its derivatives,
$a\xi \nu_\xi(\xi)$ is also $\bigo(\xi^{-2})$.  Thus, we have a
consistent expansion, and the leading order $\xi^{-2}$ term in
$\nu(\xi)$ could be obtained if needed.  Again, one can check that the
omitted terms in the expansion are all $\bigo(\xi^{-3})$, and have
been put on the right-hand-side of this last equation.  Returning to
the $Q$ variable, we have \eqref{eqn:Asymptotics2ndOrder}:
\begin{equation*}
  Q \approx A_\pm |\xi|^{\frac{ -1}{2\sigma}} \left(1 +
    \frac{b}{2\sigma a \xi}\right) \exp\set{\frac{ -i}{a}\ln |\xi| +
    \frac{i b }{a^2\xi} },
\end{equation*}
and the corrections in the phase and amplitude are at
$\bigo(\xi^{-2})$.

\section{Details of the Numerical Methods}
\label{a:numerics}

Here, we report details of our numerical scheme for solving
\eqref{QProfile}.

\subsection{Far Field Boundary Conditions}

To numerically solve \eqref{QProfile}, we restrict to the domain
$[-\ximax,\ximax]$, and impose approximate boundary conditions at
$\pm \ximax$.  For $\ximax$ large enough, $Q$ is approximated by
\eqref{eqn:Asymptotics2ndOrder} (see Proposition
\ref{prop:Asymptotics}).  This allows us to write linear Robin
conditions.  Writing $Q$ in terms of its amplitude and phase,
$Q = A e^{i \phi}$, and also in terms of its real and imaginary parts
$Q = u + i v$.  Then
\begin{equation} \label{e:phixi} \phi_\xi = \frac{-v u_\xi + u
    v_\xi}{u^2 + v^2}, \; \; A_\xi = \frac{u u_\xi + v
    v_\xi}{\sqrt{u^2+v^2}}.
\end{equation}
Using \eqref{eqn:Asymptotics2ndOrder}, we have that, as
$\xi \to \pm\infty$,
\begin{equation}
  \phi_\xi  \approx -\frac{1}{a\xi} - \frac{b}{ a^2 \xi^2}, \; \;
  \frac{A_\xi}{A}  \approx -\frac{1}{2\sigma \xi}-
  \frac{b}{2a\sigma \xi^2}.
\end{equation}
Defining
$\alpha(\xi) \equiv \frac{1}{2\sigma \xi}+ \frac{b}{2a\sigma
  \xi^2},\;\; \beta(\xi) \equiv \frac{1}{a\xi} + \frac{b}{ a^2 \xi^2}$
and substituting in \eqref{e:phixi} we obtain for large $|\xi|$,
\begin{equation}
  \label{e:farfield_bc}
  u_\xi + \alpha(\xi) u - \beta(\xi) v \approx 0, \;\;
  v_\xi +\beta(\xi) u + \alpha(\xi) v\approx 0,
\end{equation}
and thus the boundary conditions at $\pm \ximax$.

\subsection{Rescaling of the domain}

As seen in Section \ref{sect:Observations}, the turning points of
\eqref{QProfile} are located at $\xi_-=-\eps/a$ and
$\xi_+=(4-\eps)/a$.  In order to be in the asymptotically linear
regime where \eqref{e:farfield_bc} is valid, we need $\ximax$ to
exceed $\abs{\xi_\pm}$.  This presents a problem numerically, since
$\xi_\pm \to \pm \infty$, as $\sigma\to 1$.  We thus rescale the
domain, so that the turning points, in the rescaled coordinate system
remain in a fixed domain.  Setting $x= a\xi$, \eqref{QProfile} becomes
\begin{equation}
  \label{e:Qrescaled}
  a^2 Q_{xx} - Q + ia \paren{\tfrac{1}{2\sigma} Q + xQ_x} - i a b Q_x +
  i a \abs{Q}^{2\sigma} Q_x=0
\end{equation}
and boundary conditions \eqref{e:farfield_bc}, evaluated at $\xmax$,
are
\begin{equation} \label{e:Qfarfield_rescaled} 0 = u_x + \alpha(x) u -
  \beta(x) v,\,\; 0 = v_x + \beta(x) u +\alpha(x) v,
\end{equation}
with
$ \alpha(x) \equiv\frac{1}{2\sigma x} + \frac{b}{2\sigma x^2}, \;
\beta(x) \equiv\frac{1}{ax} + \frac{b}{ax^2}$.
In these coordinates, the turning points are at $x_- = -\eps$ and
$x_+=4-\eps$. Eq. \eqref{e:Qrescaled} is singular as $\sigma\to 1$
since $a\to 0$.  However, we find this to be more effective, as it
allows us to compute on a domain of fixed size for all values of
$\sigma$.

\subsection{Numerical Implementation of the Boundary Value Problem}
\label{SectSetupBVP}
We solve for $Q$ using the default Newton solver in
\cite{petsc-efficient,petsc-web-page}, along with a
sparse direct linear solver.  Due to the
condition that the maximum of the profile occurs at the origin, an
interior point of $(-\xmax, \xmax)$, we introduce the variable
$W(x) = Q(-x)$, and study $Q$ and $W$ on $(0,\xmax)$, with $Q$ and
$W$ coupled by a continuity condition at the origin.  $W$ then solves
the equation
\begin{equation}
  \label{e:Wrescaled2}
  a^2 W_{xx} - W + ia \paren{\tfrac{1}{2\sigma} W + xW_x} + i a b W_x -i a \abs{W}^{2\sigma} W_x=0.
\end{equation}
Setting $W = f + i g$, the boundary conditions
\eqref{e:Qfarfield_rescaled} are
\begin{align*}
  0 & = -f_x + \alpha(-\xmax) f - \beta(-\xmax) g,\\
  0 & = -g_x + \beta(-\xmax) f +\alpha(-\xmax) g.
\end{align*}
We now solve these equations on a uniform mesh of $N+1$ mesh points on
$[0, \xmax]$ to obtain $(u_j, v_j, f_j, g_j)_{j=0}^{N}$ along with $a$
and $b$. First and
second derivatives are approximated by second order centered finite
differences.  For instance, the real part of \eqref{e:Qrescaled} becomes
\begin{equation}
\label{e:Qdxre}
\begin{split}
&\tfrac{a^2}{\Delta x^2} \left( u_{j+1} - 2 u_j  + u_{j-1}\right) -
u_j  \\
&\quad +a\left( \tfrac{1}{2\sigma} v_j + (x_j-b + (u_j^2 + v_j^2)^{\sigma}) \tfrac{1}{2\Delta
    x}(v_{j+1} - v_{j-1}) \right) = 0
\end{split}
\end{equation}
After an analogous discretization, the farfield boundary conditions provide the needed values of
$(u_{N+1}, v_{N+1}, f_{N+1}, g_{N+1})$ for evaluating equations like
\eqref{e:Qdxre} at $j = N$.  
 
In addition to the farfield conditions, we impose symmetry and
anti-symmetry conditions at the origin,
\begin{equation*}
u_{-1} = u_1, \quad v_{-1} =g_1, \quad f_{-1} = f_{1}, \quad g_{-1} = v_{1}
\end{equation*}
the auxiliary continuity conditions at the origin,
\begin{equation*}
u_0 = f_0, \quad v_0 = g_0,
\end{equation*}
and the zero phase condition, $v_0=0$.

This system of $4 \times (N+1) + 2$ unknowns  is then solved with
$\xmax = 25$.  Solutions with $N = 1.28\times 10^6, 2.56\times
10^6, 5.12 \times 10^6$,  were obtained. As a convergence criterion,
we sought to ensure that we had good pointwise relative error,
terminating when
\begin{gather}
\abs{u_0-f_0}\leq {\rm TOL}, \quad \abs{v_0-g_0}\leq {\rm TOL},
\frac{\abs{\text{\eqref{e:Qdxre}}} }{\abs{u_j+i v_j}}\leq {\rm TOL},
\end{gather}
and an analogous equation for the imaginary counterpart to
\eqref{e:Qdxre}, along with the $f_j$ and $g_j$ equations.  We solved
with ${\rm TOL}=10^{-6}$. We compared our results
against those obtained using {\tt BVP\_SOLVER-2}, a successor to {\tt
  BVP\_SOLVER}, \cite{Shampline:2006aa, Boisvert:2012aa}.  They were
found to be in agreement, but we found that {\tt BVP\_SOLVER-2} was
unable to solve for values of $\sigma$ below 1.07, motivating us to
switch algorithms.

\subsection{Continuation method} 
As is the case for all Newton solvers, it is essential to   provide a good initial guess.  We
use the solution obtained from the time integration of gDNLS
$\sigma = 2$, and perform a continuation in $\sigma$, to solve for $Q$
at smaller values.  The initial guess for $\sigma=2$ was constructed
using the dynamic rescaling, \cite{Liu:2013ej}.  Next, we construct a
decreasing sequence of values of $\sigma$,
$ 2=\sigma_0 > \sigma_1> \ldots >\sigma_j>\ldots$,  using the
solution at $\sigma_{j-1}$ as the starting guess for solving the
solution at $\sigma_j$.  Starting with $\Delta \sigma = 0.2$, we reduce
$\sigma$ by $\Delta \sigma$,  halving the size of $\Delta \sigma$ when
the Newton solver fails.   Below $\sigma = 1.1$, the largest value of
$\Delta \sigma = 0.0125$, and for values close
to the smallest resported value of $\sigma=1.044$, $\Delta \sigma =
0.00078125$.  

\subsection{Richardson Extrapolation}

Since this discretization is second order, we contend that
quantities such as $a$ and $b$ should be $\bigo(\Delta x^2)$.  Thus,
we improve upon our results at different resolutions via Richardson
extrapolation, {\it i.e.},
$a^{{\rm Rich.}}(\sigma) = \tfrac{4}{3}a^{\Delta x/2}(\sigma) -\tfrac{1}{3}a^{\Delta x}(\sigma).$
This requires having values of the desired quantities available at the
same values of $\sigma$.  We use cubic spline interpolation to obtain
values on common $\sigma$ values.

\subsection{Limitations}

One limitation we found to our  numerical calculations is due to the singular nature
of the equation.  Recall that we expect $a\to 0$ as $\sigma
\to 1$.  Since $a$ corresponds to a length scale in \eqref{e:Qrescaled}, we should have $\Delta x \ll
a$.  Thus, as $\sigma$ tends to 1 and $a$ tends to zero, a
consistent numerical discretization requires ever smaller values of
$\Delta x$.  This limited us to values of $\sigma$ near 1.04.  Also,
as shown in Figure \ref{fig:AmpPlusMultiTurning}, computed values of
$Q$, in the tail, are reaching the limit of double precision
floating point, as the values are smaller than $10^{-300}$.

\bibliographystyle{plain}

\bibliography{dnls_refs}

\begin{thebibliography}{10}

\bibitem{Ambrose:2015en}
David~M Ambrose and Gideon Simpson.
\newblock {Local existence theory for derivative nonlinear Schr\"odinger
  equations with noninteger power nonlinearities}.
\newblock {\em SIAM Journal on Mathematical Analysis}, 47(3):2241--2264, 2015.

\bibitem{Anderson:1983wg}
D.~Anderson and M.~Lisak.
\newblock {Nonlinear asymmetric self-phase modulation and self-steepening of
  pulses in long optical waveguides}.
\newblock {\em Phys. Rev. A}, 27(3):1393--1398, 1983.

\bibitem{petsc-web-page}
S.~Balay, S.~Abhyankar, M.F. Adams, J.~Brown, P.~Brune, K.~Buschelman,
  L.~Dalcin, V.~Eijkhout, W.D. Gropp, D.~Kaushik, M.G. Knepley, L.~Curfman
  McInnes, K.~Rupp, B.F. Smith, S.~Zampini, and H.~Zhang.
\newblock {PETS}c {W}eb page.
\newblock \url{http://www.mcs.anl.gov/petsc}, 2015.

\bibitem{petsc-efficient}
S.~Balay, W.~D. Gropp, L.~Curfman McInnes, and B.~F. Smith.
\newblock Efficient management of parallelism in object oriented numerical
  software libraries.
\newblock In E.~Arge, A.~M. Bruaset, and H.~P. Langtangen, editors, {\em Modern
  Software Tools in Scientific Computing}, pages 163--202. Birkh{\"{a}}user
  Press, 1997.

\bibitem{Boisvert:2012aa}
J.~J. Boisvert, P.~H. Muir, and R.~J. Spiteri.
\newblock A numerical study of global error and defect control schemes for
  bvodes.
\newblock Technical Report,
  \url{http://cs.stmarys.ca/~muir/BVP_SOLVER_Webpage.shtml}, 2012.

\bibitem{BCR}
C.~J. Budd, S.~Chen, and R.~Russell.
\newblock {New self-similar solutions of the nonlinear Schr{\"o}dinger equation
  with moving mesh computations}.
\newblock {\em J. Comput. Phys.}, 152(2):756--789, 1999.

\bibitem{HO16}
Masayuki Hayashi and Tohru Ozawa.
\newblock {Well-posedness for a generalized derivative nonlinear
  Schr{\"o}dinger equation}.
\newblock \href{http://arxiv.org/abs/1601.04167}{arXiv:1601.04167}, 2016.

\bibitem{Hayashi:1993vj}
N.~Hayashi.
\newblock {The initial value problem for the derivative nonlinear
  Schr{\"o}dinger equation in the energy space}.
\newblock {\em Nonlinear Analysis. Theory, Methods \& Apllications},
  20(7):823--833, 1993.

\bibitem{Hayashi:1992wl}
N.~Hayashi and T.~Ozawa.
\newblock {On the derivative nonlinear Schr{\"o}dinger equation}.
\newblock {\em Physica D: Nonlinear Phenomena}, 55(1-2):14--36, 1992.

\bibitem{Kaup1978}
D.J. Kaup and A.~C Newell.
\newblock {An exact solution for a derivative nonlinear Schr{\"o}dinger
  equation}.
\newblock {\em J. Math. Phys.}, 19(4):798--801, 1978.

\bibitem{Liu:2015aa}
J.~Liu, P.~Perry, and C.~Sulem.
\newblock {Global Existence for the Derivative Nonlinear Schr{\"o}dinger
  Equation by the Method of Inverse Scattering}.
\newblock \href{http://arxiv.org/abs/1511.01173}{arXiv:1511.01173}, 2015.

\bibitem{Liu:2013ej}
X.~Liu, G.~Simpson, and C.~Sulem.
\newblock {Focusing singularity in a derivative nonlinear Schr{\"o}dinger
  equation}.
\newblock {\em Physica D}, 262:48--58, 2013.

\bibitem{Liu:2013cq}
X.~Liu, G.~Simpson, and C.~Sulem.
\newblock {Stability of Solitary Waves for a Generalized Derivative Nonlinear
  Schr{\"o}dinger Equation}.
\newblock {\em J. Nonlinear Sci.}, 23(4):557--583, 2013.

\bibitem{McPSS}
D.W. McLaughlin, G.C. Papanicolaou, C.~Sulem, and P.-L. Sulem.
\newblock The focusing singularity of the cubic schr{\"o}dinger equation.
\newblock {\em Phys. Rev. A}, 34:1200--1210, 1986.

\bibitem{MRS2011}
F.~Merle, P.~Rapha{\"e}l, and J.~Szeftel.
\newblock Stable self similar blow up for slightly $l^2$ super-critical nls
  equations.
\newblock {\em Geom. Funct. Anal.}, 20(4):1028--1071, 2010.

\bibitem{Moses:2007vv}
J.~Moses, B.~A. Malomed, and F.~W. Wise.
\newblock {Self-steepening of ultrashort optical pulses without
  self-phase-modulation}.
\newblock {\em Phys. Rev. A}, 76(2), 2007.

\bibitem{Pelinovsky:2015aa}
D.~Pelinovsky and Y.~Shimabukuro.
\newblock {Existence of global solutions to the derivative NLS equation with
  the inverse scattering transform method}.
\newblock Preprint, 2015.

\bibitem{SanchezArriaga:2010jy}
G.~Sanchez-Arriaga, D.~Laveder, T.~Passot, and P.-L. Sulem.
\newblock {Quasicollapse of oblique solitons of the weakly dissipative
  derivative nonlinear Schrodinger equation}.
\newblock {\em Phys. Rev. E}, 82(1):016406, 2010.

\bibitem{Shampline:2006aa}
L.~F. Shampline, P.~H. Muir, and H.~Xu.
\newblock A user-friendly fortran bvp solver.
\newblock {\em JNAIAM}, 1(2):201--217, 2006.

\bibitem{Sulem:1999kx}
C.~Sulem and P.-L. Sulem.
\newblock {\em {The Nonlinear Schr{\"o}dinger Equation: Self-Focusing and Wave
  Collapse}}.
\newblock Springer, 1999.

\bibitem{Tzoar:1981vq}
N.~Tzoar and M.~Jain.
\newblock {Self-phase modulation in long-geometry optical waveguides}.
\newblock {\em Phys. Rev. A}, 23(3):1266--1270, 1981.

\bibitem{Wu:2014uc}
Y.~Wu.
\newblock {Global well-posedness on the derivative nonlinear Schr{\"o}dinger
  equation revisited}.
\newblock {\em arxiv.org 1404.5159v3}, April 2014.

\end{thebibliography}

\end{document}